\documentclass{amsart}

\usepackage{intro}
\usepackage{cite}
\usepackage{hyperref, mathtools, tikz-cd}
\usepackage{url}

\newcommand{\height}{\operatorname{ht}}
\newcommand{\symb}{\operatorname{symb}}

\newcommand{\dgMod}{\operatorname{dgMod}}
\newcommand{\colim}{\operatorname{colim}}

\DeclarePairedDelimiter{\ceiling}{\lceil}{\rceil}

\title{Cohomology of $p$-adic Chevalley groups}
\author{Andrea Dotto}
\address{
    Department of Mathematics,
    King's College London,
    Strand, London, WC2R 2LS, United Kingdom
}
\email{andrea.dotto@kcl.ac.uk}

\author{Bao V.~Le Hung}
\address{Department of Mathematics,
Northwestern University, 
2033 Sheridan Road, 
Evanston, Illinois 60208, USA}
\email{lhvietbao@googlemail.com}

\begin{document}
\maketitle

\begin{abstract}
    Let~$G$ be a split connected reductive group over the ring of integers of a finite unramified extension~$K$ of $\bQ_p$.
    Under a standard assumption on the Coxeter number of~$G$, 
    we compute the cohomology algebra of~$G(\cO_K)$ and its Iwahori subgroups, with coefficients in the residue field of~$K$.
    Our methods involve a new presentation of some graded Lie algebras appearing in Lazard's theory of saturated $p$-valued groups, 
    and a reduction to coherent cohomology of the flag variety in positive characteristic.
    We also consider the case of those inner forms of $\GL_n(K)$ that give rise to the Morava stabilizer groups in stable homotopy theory.
\end{abstract}

\tableofcontents

\section{Introduction.}
Let~$K/\bQ_p$ be a finite extension, with ring of integers~$\cO_K$, residue field~$k \cong \bF_{p^f}$, and ramification degree~$e$.
Let~$G$ be a split connected reductive group over~$\cO_K$.
If~$\Gamma \subset G(\cO_K)$ is a compact open subgroup, the (continuous) cohomology algebra $H^*(\Gamma, k)$ is well-understood whenever $\Gamma$ is small enough,
by the foundational work of Lazard~\cite{Lazard}: more precisely, if $\Gamma$ is a uniform pro-$p$ group, then 
\[
H^*(\Gamma, k) = \wedge^* H^1(\Gamma, k)
\]
is an exterior algebra generated in degree one~\cite[Proposition~V.2.5.7.1]{Lazard}.
However, many subgroups of~$G(\cO_K)$ of significant interest in number theory, starting from~$G(\cO_K)$ itself, are not uniform groups,
or even pro-$p$. In particular, the cohomology of $G(\cO_K)$ with mod~$p$ coefficients remains unknown, outside a few scattered examples.

Our main result in this paper is the computation of $H^*(G(\cO_K), k)$ whenever~$e = 1$ and~$p > h+1$, where~$h$ is the maximum of the Coxeter numbers
of the irreducible components~$\Phi_1, \ldots, \Phi_n$ of the root system~$\Phi$ of~$G$.
To state it, we will need the following notation.
To each component~$\Phi_i$ one can associate a finite set of positive integers, called the exponents, whose definition is recalled in Section~\ref{Lie algebra cohomology}.
For example, if~$\Phi_i$ has type~$A_n$ then the exponents are the positive integers from~$1$ to~$n$.
We write~$\Lambda_{k}(\Phi_i)$ for the free graded-commutative $k$-algebra on generators of degree~$2m+1$, where~$m$ runs through the exponents of~$\Phi_i$, 
and we write $\Lambda_k(\Phi) := \otimes_{i=1}^n \Lambda_{k}(\Phi_i)$
for the graded-commutative tensor product.
We also write~$Z$ for the connected centre of~$G$, and~$Z_1$ for the pro-$p$ Sylow subgroup of~$Z(\cO_K)$.
Then the result is as follows.

\begin{thm}\label{main theorem intro}
With notation as in the previous paragraph, assume that~$e = 1$ and~$p > h+1$.
Let $I \subset G(\cO_K)$ be an Iwahori subgroup.
Then the restriction map $H^*(G(\cO_K), k) \to H^*(I, k)$ is an isomorphism, and 
\[
H^*(I, k) \cong H^*(Z_1, k) \otimes_k \Lambda_k(\Phi)^{\otimes f}
\]
as graded $k$-algebras.
\end{thm}

\begin{rk}\begin{enumerate}
\item
A version of Theorem~\ref{main theorem intro} continues to hold after replacing~$k$ with an irreducible smooth representation $\sigma$ of $k[G(\cO_K)]$ with $p$-small highest weight,
and will be contained in our work in progress~\cite{DLW}.
This is motivated by expectations from the categorical mod~$p$ Langlands program, which we briefly explain in the case of $\GL_{n, \bQ_p}$. 
For generic $\sigma$, there is a $k$-algebra map 
\[H^*(\GL_n(\bZ_p), \sigma^\vee \otimes \sigma) \to \mathrm{Ext}^*_{\GL_n(\bQ_p)}(c\text{-Ind}_{\GL_n(\bZ_p)}^{\GL_n(\bQ_p)} \sigma)\]
identifying the left-hand side with a certain ``weight $0$ part'' of the Hecke $\Ext$-algebra on the right-hand side.
After taking the central character into account, the categorical mod~$p$ Langlands conjecture predicts 
that the Hecke $\Ext$-algebra is isomorphic to the $\Ext$-algebra of a certain line bundle on the irreducible component $C_\sigma$ of the Emerton--Gee stack~$\cX_{n, \bQ_p}$, 
which we compute explicitly in~\cite{DLW}.
In fact, we first arrived at Theorem~\ref{main theorem intro} by contemplating the weight $0$ part of this computation.
\item As alluded to before, if we set $K_1:=\ker( G(\cO_K)\to G(k))$, then Lazard's theory gives a simple description of $H^*(K_1,k)$, and a naive approach 
to Theorem \ref{main theorem intro} would be to use the Hochschild--Serre spectral sequence for $K_1\triangleleft\, G(\cO_K)$. However, since the finite group $G(k)$ has 
infinite cohomological dimension, there must be many non-zero differentials in the $E_2$ page, and it seems intractable to understand the resulting large amount of cancellations.
\end{enumerate}
\end{rk}

Theorem~\ref{main theorem intro} is the conjunction of Propositions~\ref{pp: Iwahori cohomology for reductive groups} and~\ref{pp: cohomology of the maximal compact subgroup},
whose main input is Theorem~\ref{main theorem}, which is the computation of $H^*(I, k)$ when~$G$ is almost-simple simply connected.
We now sketch the proof of Theorem~\ref{main theorem}, assuming for simplicity that~$K = \bQ_p$.
Recall that the de Rham cohomology of any compact connected real Lie group with the same root system as~$G$ is isomorphic to~$\Lambda_\bR(\Phi)$,
as can be seen identifying it with the cohomology of $\fg:= \Lie(G)$.
This coincidence can be explained by the method of proof of Theorem~\ref{main theorem intro}, which also uses Lie algebra cohomology; however, rather than
working with the simple Lie algebra~$\fg$, we will be working with a finite-dimensional nilpotent $\bF_p$-Lie algebra~$\lbar \fg$, which arises naturally in Lazard's theory,
and is defined in~\eqref{definition of gbar}.
If $T \subset G$ is a maximal torus, there is a natural action of~$T$ on~$\lbar \fg$, i.e.\ a grading by weights.
Via the spectral sequence~\eqref{Lazard spectral sequence}, which is constructed in~\cite{SorensenHochschild}, 
we reduce the proof of Theorem~\ref{main theorem intro} to the computation of the weight zero part of $H^*(\lbar \fg, \bF_p)$.
We do this in two steps.

The first step is carried out in Section~\ref{graded Lie algebras}, which works with arbitrary~$e \geq 1$, and computes a novel presentation of~$\lbar \fg$:
in more detail, when~$e = 1$, $\lbar \fg$ is proved to be isomorphic to the semidirect product of
the nilpotent upper-triangular subalgebra $\fn \subset \fg$ with its adjoint representation on $\fg/\fn$.
As part of the proof of this isomorphism, in Theorem~\ref{graded MP} we 
compute the associated graded of the 
Moy--Prasad filtration 
on the pro-$p$ Sylow subgroup $I_1 \subset I$,
in terms of nilpotent loop Lie algebras.
This is a refinement of the Moy--Prasad isomorphism for split groups, which is likely to be known to experts, 
but does not seem to be documented in the literature in this form.
Its relevance to the problem at hand comes from the fact that the Moy--Prasad filtration is a saturated $p$-valuation whenever
$he < p-1$ (this is essentially the content of~\cite[Proposition~3.4]{LahiriSorensen}).

In the second step, we use our presentation of~$\lbar \fg$ to identify~$H^*(\lbar \fg, \bF_p)^T$ with the weight zero part of the $\fn$-cohomology of $\wedge(\fg/\fn)^\vee$,
and then with the cohomology of the coherent $\cO_X$-algebra on $X := G/B$ associated to $\wedge(\fg/\fn)^\vee$.
This is closely related to the Hodge cohomology of~$X$, and so to the Chow ring $A(G/B) \otimes \bF_p$. 
By work of Demazure~\cite{Demazureinvariants, DemazureSchubert}, there is a graded $\bF_p$-algebra surjection
\[
\Sym \ft^\vee \to A(G/B) \otimes \bF_p,
\]
sometimes called the Borel presentation, whose kernel is generated by a regular sequence of homogeneous elements of degree~$m_i+1$.
Using this, we prove in Section~\ref{cohomology} that $H^*(\lbar \fg, \bF_p)^T$ is computed by the derived tensor product
\[
(A(G/B) \otimes \bF_p) \otimes^L_{\Sym (\ft^\vee[-2])} \bF_p[0]
\]
after doubling the degrees on the Chow ring.
Then Theorem~\ref{main theorem intro} follows on computing this tensor product with a Koszul complex. 

\begin{rk}
As part of the proof of Theorem~\ref{main theorem intro}, we prove that the spectral sequence~\eqref{Lazard spectral sequence}
degenerates at~$E_1$ for the group~$I_1$, after passing to $T$-invariants (provided that~$p>h+1$).
Given the computation of $H^*(\lbar \fg, \bF_p)^T$, the proof is an application of the Lazard isomorphism in characteristic zero.
As we were writing this paper, we learned that Ardakov has proved the degeneration of~\eqref{Lazard spectral sequence} for all saturated $p$-valued groups~\cite{Ardakovsaturable}, 
under a similar constraint on~$p$, 
by different methods. 
\end{rk} 

\begin{rk}
The bound~$p > h+1$ in Theorem~\ref{main theorem intro} is essentially optimal with respect to the methods of this paper, 
since the pro-$p$ Sylow subgroup of~$I$ need not admit a saturated $p$-valuation if $p \leq h+1$.
The structure of $H^*(I, k)$ is furthermore expected to change significantly once~$p \leq h+1$.
For example, the pro-$p$ Iwahori subgroup of $\GL_n(\bQ_p)$ has nontrivial $p$-torsion if~$p \leq n+1$, hence $\GL_n(\bZ_p)$ has an irreducible $\bF_p$-representation
with infinite-dimensional cohomology; and indeed it turns out that $H^*(\GL_{p-1}(\bZ_p), \bF_p)$ is infinite-dimensional whenever~$p$ is odd~\cite[Theorem~1.10]{Henncentralizers}.
\end{rk}

\begin{rk}
We expect that the methods of Section~\ref{graded Lie algebras} should apply to arbitrary connected reductive groups over~$\bQ_p$, and the methods of Section~\ref{cohomology}
should apply to unramified groups over~$\bQ_p$ and their inner forms.
As an example, motivated by applications in algebraic topology, we consider in Section~\ref{Morava} a central division algebra $D/\bQ_p$
of invariant~$1/n$, and we compute $H^*(\cO_D^\times, \bF_{p^n})$ for~$p > n+1$.
The group~$\cO_D^\times$ is closely related to the Morava stabilizer group of height~$n$ in stable homotopy theory, whose $\bF_{p^n}$-cohomology has 
recently been studied in~\cite{KangSalch}, by different methods which involve lifting to characteristic zero, and so require more stringent assumptions than $p > n+1$
(namely, that~$p$ should be larger than an inexplicit function of~$n$).
\end{rk}

\subsection{Acknowledgments.}
AD would like to thank Konstantin Ardakov, for sharing his results about the spectral sequence~\eqref{Lazard spectral sequence}, and Beth Romano, 
for useful discussions about graded Lie algebras, and comments on an earlier draft of this note.
AD was supported by a Royal Society University Research Fellowship.
 
BLH would like to thank Michael Larsen for useful discussions on the question studied in this note.
BLH ~acknowledges support from the National Science Foundation under grants Nos.~DMS-1952678 and DMS-2302619 and the Alfred P.~Sloan Foundation.

Some of the ideas in Section~\ref{cohomology} were first conceived during our on-going joint work \cite{DLW} with Junho Won, and we want to thank him for the collaboration.
We are also grateful to Toby Gee for comments on an earlier draft of this note.

\section{Background.}\label{background}
Let~$K/\bQ_p$ be a finite extension, with ring of integers~$\cO_K$, residue field~$k \cong \bF_{p^f}$, and ramification degree~$e$.
We write~$v$ for the unique surjective valuation $v: K^\times \to \bZ$ and fix a uniformizer~$\varpi$ of~$\cO_K$.
When~$e = 1$, we let~$\varpi = p$.
If~$i, j$ are numbers, the symbol~$\delta_{i, j}$ denotes~$1$ if~$i = j$ and~$0$ if~$i \ne j$.
Similarly, if~$P$ is a statement, the symbol~$\delta_P$ denotes~$1$ if~$P$ is true and~$0$ if~$P$ is false.

\subsection{Reductive groups.}
Let~$G/K$ be an almost-simple, simply connected group with a Borel subgroup~$B$ and a split maximal torus $T \subset B$.
In Section~\ref{reductive groups} we will allow~$G$ to be a split reductive group over~$K$.
Let~$\fg_K$ be the $K$-Lie algebra of~$G(K)$, and similarly let~$\ft_K := \Lie(T(K)), \fb_K := \Lie(B(K))$.
We write~$\Phi \subset X^*(T)$ for the roots, $\Phi^+ \subset \Phi$ for the $B$-positive roots, $\Delta \subset \Phi^+$ for the simple roots, $W$ for the Weyl group of~$T$ in~$G$,
and~$\ell$ for the length function on~$W$ corresponding to~$\Delta$.
We write $\height : \Phi \cup \{0\} \to \bZ$ for the height function corresponding to~$\Delta$, 
normalized so that $\height(-\alpha) = -\height(\alpha)$ for all~$\alpha \in \Phi^+$, and $\height(0) = 0$. 
We write~$h$ for the Coxeter number of~$\Phi$, so that $h = \height(\alpha_{\max}) + 1$, where~$\alpha_{\max} \in \Phi^+$ is the highest root.

If~$\alpha \in \Phi$, we write~$U_\alpha$ for the corresponding root subgroup, and if $\alpha \in \Phi \cup \{0\}$, we write 
$\fg_\alpha \subset \fg$ for the corresponding eigenspace.
We fix a Chevalley system~$(G, B, T)$ over~$K$, i.e.\ a choice of nonzero $X_\alpha \in \fg_\alpha$ for all~$\alpha \in \Delta$ satisfying the properties in~\cite[Definition~2.9.8(1)]{KPbook}. 
For all~$\alpha \in \Phi$, this gives rise to an isomorphism
\[
u_\alpha: \Ga \to U_\alpha
\]
defined over~$K$, and such that $du_\alpha(1) = X_\alpha$.
Composing~$u_\alpha^{-1}$ with~$v$ we obtain a valuation of the root datum of~$G$, i.e.\ a collection of maps
\[
v_\alpha: U_\alpha(K) \to \bZ \cup \{+\infty\}
\]
satisfying the axioms in~\cite[Definition~6.1.2]{KPbook}.
We write~$U_\alpha(\varpi^i\cO_K) := v_\alpha^{-1}[i, +\infty]$.
The standard apartment of~$G$, denoted~$\cA$, is the set of all valuations equipollent to~$v_\alpha$ in the sense of~\cite[Definition~6.1.7]{KPbook}.
It is an affine space under~$V = X_*(T) \otimes_{\bZ} \bR$, and the choice of~$v_\alpha$ specifies an isomorphism $\cA \cong V$.
The affine root system of~$G$ becomes identified with $\Psi = \{\alpha+n : \alpha \in \Phi, n \in \bZ\}$, viewed as a set of affine functions $\cA \to \bR$.
We write~$C$ for the dominant base alcove of~$V$ under~$\Psi$.
As usual, we write $\rho := \frac 1 2\sum_{\alpha \in \Phi^+}\alpha$ for the half-sum of positive roots,
and $w \cdot \lambda := w(\lambda+\rho) - \rho$ for all~$\lambda \in X^*(T)$.

\subsection{Lie algebras.}\label{Lie algebras}
By~\cite[Remark~2.9.9]{KPbook}, the Chevalley system~$\{X_\alpha: \alpha \in \Phi\}$ extends to a Chevalley basis for~$\fg_K$.
In more detail, to each root~$\alpha$ we associate $H_\alpha := d \alpha^\vee(1) \in \ft_K$.
Then $\{H_\alpha: \alpha \in \Delta\}$ is a $K$-basis of $\ft_K$, and we have the structure constants
\begin{gather*}
[H_\alpha, X_\beta] = \langle \alpha^\vee, \beta \rangle X_\beta\\
[X_\alpha, X_{-\alpha}] = H_{\alpha}\\
[X_\alpha, X_\beta] = N_{\alpha, \beta} X_{\alpha + \beta} \text{ if $\alpha, \beta \in \Phi$ are not proportional,}
\end{gather*}
for some $N_{\alpha, \beta} \in \bZ$
such that $|N_{\alpha, \beta}| = r+1$, where~$r$ is the largest integer for which $\beta-r\alpha \in \Phi$.
The Chevalley basis defines a $\bZ$-form~$\fg_\bZ$ of~$\fg_K$ with triangular decomposition $\fg_\bZ = \fn_\bZ^{-} \oplus \ft_\bZ \oplus \fn_\bZ$.
As usual, we put $\fb_\bZ := \ft_\bZ \oplus \fn_\bZ, \fb_{\bZ}^{-} := \ft_\bZ \oplus \fb_{\bZ}^-$.
We write $\fg_k := \fg_\bZ \otimes_{\bZ} k$, and similarly for $\fn_k, \fn_k^{-}, \ft_k, \fb_k, \fb_k^{-}$.
When there is no danger of confusion, we will often drop~$k$ from the notation (but note that $\fg$ denotes the split $\bF_p$-form of~$\fg_k$ in Section~\ref{cohomology}).

We define $\fg[v] := \fg \otimes_{k} k[v]$, where~$v$ is an indeterminate, and we define~$\tld \fg$ as the preimage of~$\fn$ under the reduction map $\fg[v] \to \fg$.
Then
\[
\tld \fg = v\fn^{-}[v] \oplus v \ft[v] \oplus \fn[v].
\]
Recall that $\fg$ admits a $\bZ/h\bZ$-grading, sometimes called the Coxeter grading, defined by 
\[
\fg[i] := \bigoplus_{\height(\alpha) \equiv i \operatorname{mod} h} \fg_{\alpha}.
\]
For any~$e>0$, we define a $\frac{1}{he}\bZ$-grading on~$\fg[v]$ by
\begin{gather*}
\deg(X_\alpha \otimes v^i) := \height(\alpha)/he + i/e \text{ if } \alpha \in \Phi\\
\deg(H \otimes v^{i}) := i/e \text{ if } H \in \ft.
\end{gather*}
Then $\tld \fg$ is a positively graded $k$-Lie subalgebra of~$\fg[v]$. 
We denote it by~$\tld \fg_e$.
There is furthermore a $T$-action on $\fg[v] = \fg \otimes_k k[v]$, namely the tensor product of the adjoint action on~$\fg$ with the trivial action on~$k[v]$.
Then $\tld \fg \subset \fg[v]$ is a $T$-submodule.

Observe that $\fg[v]$ and $\tld \fg_e$ are finite free $k[v]$-modules of the same rank, and they are $k[v]$-Lie algebras 
(i.e.\ the bracket is $k[v]$-bilinear).
Let~$\lambda \in k^\times$ be the image of~$p/\varpi^e$.
We will write~$\epsilon : \tld \fg_e \to \tld \fg_e$ for the degree-one operator of multiplication by~$\lambda v^e$.
The following notation will sometimes be useful: if $i \geq 1$, and $\alpha$ is a root such that $\height(\alpha) \equiv i \mod h$,
then we write
\[
i_\alpha := (i-\height(\alpha))/h \in \bZ_{\geq 0}.
\]

\begin{lemma}\label{alternative presentation of tld g_e}
Let~$i \geq 1$. 
If~$h \nmid i$, then
$\{X_\alpha \otimes v^{i_\alpha}: \height(\alpha) \equiv i \mod h\}$ is a $k$-basis of $\gr^{i/he} \tld \fg_e$.
If~$h \mid i$, then $\{H_\alpha \otimes v^{i/h} : \alpha \in \Delta\}$ is a $k$-basis of $\gr^{i/he} \tld \fg_e$.
\end{lemma}
\begin{proof}
This is a direct computation.
\end{proof}

\begin{corollary}\label{alternative presentation of tld g}
Assume~$e = 1$ and~$i \geq 1$.
Then there are isomorphisms $\gr^{i/h}\tld \fg_1 \to \fg[i]$ that are equivariant for the $T$-action and the Lie bracket, 
and such that
$v: \gr^{i/n} \tld \fg_1 \to \gr^{i/n+1} \tld \fg_1$ corresponds to the identity.
\end{corollary}
\begin{proof}
The isomorphisms are defined by $X_\alpha \otimes v^{i_\alpha} \mapsto X_\alpha$ when~$h \nmid i$, and $H_{\alpha} \otimes v^{i/h} \mapsto H_\alpha$ when~$h \mid i$.
\end{proof}

\begin{example}
Let~$\fg := \mathfrak{sl}_3$.
Let~$E_{i,j}$ be the elementary matrix with $1$ in the $(i,j)$-th entry.
Then $\gr^{1/3}\tld \fg_1$ has basis $\{E_{1,2}, E_{2, 3}, E_{3, 1} \otimes v\}$, $\gr^{2/3}\tld \fg_1$ has basis $\{E_{1, 3}, E_{2, 1} \otimes v, E_{3, 2} \otimes v\}$, and $\gr^1 \tld \fg_1$
has basis $\{(E_{11}-E_{22})\otimes v, (E_{22}-E_{33})\otimes v\}$. 
\end{example}

We can make analogous constructions when $\fg_K = \fz_K \oplus \fg_K^{\der}$ is the direct sum of an abelian Lie algebra and a simple Lie algebra,
by choosing a $\bZ$-lattice in~$\fz_K$, and then putting $\fg_\bZ := \fz_\bZ \oplus \fg_\bZ^{\der}$.
Corollary~\ref{alternative presentation of tld g} remains true in this context, provided we put $\deg(Z \otimes v^i) :=i/e$ for all~$Z \in \fz_k$.

\subsection{Lie algebra cohomology.}\label{Lie algebra cohomology}
Let~$\tau$ be an abelian group, and let~$\fh$ be a $\tau$-graded $k$-Lie algebra.
If~$V$ is an $\fh$-module, the Lie algebra cohomology groups $H^*(\fh, V)$ are computed by $(V \otimes_k \wedge \fh^\vee, \del)$, where the differential~$\del$ is defined
as in~\cite[Section~I.9.17]{Jantzenbook} (see also~\cite[Chapitre~II]{KoszulLie}).
If~$V$ is also $\tau$-graded, and the action of~$\fh$ satisfies
\[
\deg(H . x) = \deg(H)+\deg(x)
\]
for all~$H \in \fh, x \in V$, then~$\del$ has degree~$0$ and so $H^*(\fh, V)$ inherits a natural $\tau$-grading.
We will apply this to situations where $\tau = X^*(T)$, and we will write $H^*(\fh, V)^T$ for the degree zero part (i.e.\ the weight zero part).

The cohomology of the simple Lie algebra~$\fg_K$, with coefficients in the trivial $\fg_K$-module, can be described as follows.
Let $:= \dim_K \ft_K$, and let $m_1 \leq m_2 \leq \ldots \leq m_r \in \bZ$ be the exponents of the root system~$\Phi$ of~$\fg_K$.
By definition, this means that the degrees of the homogeneous generators of $(\Sym \ft_K^\vee)^W$ are the~$m_i+1$ (provided that
$\deg \ft_K^\vee = 1$).
Hence~$m_1 = 1$, and~$m_{r} = h-1$.
For any commutative ring~$R$ with~$1$, we write
\[
\Lambda_R(\Phi) := \wedge(x_{2m_i+1}: 1 \leq i \leq r)
\]
for the free graded-commutative $R$-algebra on generators of degree~$2m_i+1$.
More generally, if~$\Phi$ is a root system with irreducible components~$\Phi_1, \ldots, \Phi_s$, then we define $\Lambda_R(\Phi) := \otimes_{i=1}^s \Lambda_R(\Phi_i)$.
It is an exterior algebra.
We then have a graded $K$-algebra isomorphism
\begin{equation}\label{cohomology of semisimple Lie algebra}
H^*(\fg_K, K) \cong \Lambda_K(\Phi),
\end{equation}
by~\cite[Theorem~18.1]{CELie},
which remains true if~$\fg_K$ is only assumed to be semisimple.
We emphasize that~\eqref{cohomology of semisimple Lie algebra} computes the cohomology of the $K$-Lie algebra $\fg_K$ with coefficients in~$K$.
In Section~\ref{cohomology} we will also have to consider the cohomology of the $\bQ_p$-Lie algebra $\fg_K$ 
with coefficients in~$K$.
They are related by an isomorphism of graded $K$-algebras
\begin{equation}\label{Q_p-Lie to K-Lie}
H^*_{\text{$\bQ_p$-Lie}}(\fg_K, K) \cong H^*_{\text{$K$-Lie}}(\fg_K \otimes_{\bQ_p} K, K).
\end{equation}
We will often apply this in contexts where~$K/\bQ_p$ is Galois, in which case $\fg_K \otimes_{\bQ_p} K$ decomposes as a direct product $\prod_{\sigma \in \Gal(K/\bQ_p)}\fg_K^{\sigma}$,
and the $\bQ_p$-Lie cohomology can therefore be computed using the K\"unneth formula.

Let $\fn := \fn_\bZ \otimes_\bZ k$ be as in Section~\ref{Lie algebras}.
Under the assumption $p \geq h-1$ we have the characteristic~$p$ version of Kostant's theorem:
\begin{equation}\label{Kostant's theorem}
H^i(\fn, k) \cong \bigoplus_{\ell(w) = i}k(w \cdot 0)
\end{equation}
as $T$-modules, see~\cite[Theorem~2.5]{FPdiscrete}.
Hence for every~$\lambda \in X^*(T)$ we have
\begin{equation}\label{twisted Kostant's theorem}
H^i(\fn, \lambda) \cong \bigoplus_{\ell(w) = i}k(w \cdot 0 + \lambda).
\end{equation}

\subsection{Frobenius kernels.}
For any affine algebraic $k$-group $\cG$ we write~$\cG_1$ for its first Frobenius kernel.
Its Hopf algebra of functions is the $p$-restricted universal enveloping algebra of~$\Lie(\cG)$.
If~$V$ is a $\cG$-module, there is a spectral sequence
\begin{equation}\label{ordinary to restricted}
E_1^{i, j} := H^{j-i}(\Lie \cG, V) \otimes_k (\Sym^i(\Lie \cG)^\vee)^{(1)} \Rightarrow H^{i+j}(\cG_1, V),
\end{equation}
which is a version of the May spectral sequence for restricted Lie algebra cohomology,
see~\cite[Lemma~I.9.16, Proposition~I.9.20]{Jantzenbook}.
Here and in what follows, for any $k$-vector space~$W$ and any~$j \in \bZ$, the symbol~$W^{(j)}$ denotes the $k$-vector space with the same abelian group structure as~$W$, 
and scalar multiplication given by~$b \cdot v = b^{1/p^j}v$.
If~$W$ is a $\cG$-module, and $j \geq 0$, then~$W^{(j)}$ is naturally a $\cG$-module, by~\cite[Section~I.9.10]{Jantzenbook}.
For example, if~$\cG = T$ and~$j \geq 0$, then~$W^{(j)}$ is the $T$-module obtained by multiplying all weights by~$p^j$.

If~$\cG$ is reduced, the assumption of~\cite[Proposition~I.6.6]{Jantzenbook} is satisfied (by~\cite[Section~I.6.5(2), Proposition~I.9.5]{Jantzenbook}), and so 
there is also the Hochschild--Serre spectral sequence for $\cG_1 \subset \cG$, which takes the form 
\begin{equation}\label{Hochschild--Serre for Frobenius kernel}
E_{2}^{i, j}:= H^i(\cG/\cG_1, H^j(\cG_1, V)) \Rightarrow H^{i+j}(\cG, V).
\end{equation}

\subsection{Filtrations.}
Recall that a $p$-filtration on a group~$\Gamma$ is a function $\omega : \Gamma \setminus \{1\} \to (0, + \infty)$ such that, if we put $\omega(1) = +\infty$, then
\begin{enumerate}
\item $\omega(xy^{-1}) \geq \min\{\omega(x), \omega(y)\}$,
\item $\omega([x, y]) \geq \omega(x)+\omega(y)$,
\item $\omega(x^p) \geq \varphi(\omega(x)) := \min \{p\omega(x), \omega(x)+1\}$.
\end{enumerate}
See~\cite[II.1.2.10]{Lazard}, and note that we add the assumption that $\omega(x)$ is finite when~$x \ne 1$. 
We will always work with $p$-filtrations whose image~$\tau$ in~$(0, + \infty)$ has the form $n^{-1}\bZ$ for some $n \in \bZ_{>0}$, and is therefore stable under~$\varphi$.
The $p$-filtration~$\omega$ gives rise to a decreasing filtration of~$\Gamma$ by normal subgroups 
\[
\Fil^i \Gamma := \{x \in \Gamma: \omega(x) \geq i\}
\]
such that $[\Fil^i \Gamma, \Fil^j \Gamma] \subset \Fil^{i+j}\Gamma$.
We also write
\[
\Fil^{i+} \Gamma := \{x \in \Gamma: \omega(x) > i\}.
\]
If~$x \in \Gamma$, we write~$\symb(x) \in \gr \Gamma$ for the principal symbol of~$x$ with respect to~$\Fil^i$, i.e.\ the image of~$x$ in $\Fil^{\omega(x)}\Gamma/\Fil^{\omega(x)+}\Gamma$. 
Then 
\[
\gr \Gamma := \bigoplus_{i \in \tau}\Fil^i \Gamma/\Fil^{i+}\Gamma
\]
is a $\tau$-graded Lie algebra over~$\bF_p$, whose bracket is defined as follows: if $\omega(x) = i, \omega(y) = j$, then 
\[
[x \Fil^{i+}\Gamma, y\Fil^{j+}\Gamma] := xyx^{-1}y^{-1}\Fil^{(i+j)+}\Gamma.
\]
We also define a function $\epsilon : \gr \Gamma \to \gr \Gamma$ as follows: if~$ \omega(x) = i$, then
\[
\epsilon (x\Fil^{i+}\Gamma) := x^p\Fil^{\varphi(x)+}\Gamma.
\]
Then~\cite[Th\'eor\`eme~II.1.2.11]{Lazard} shows that $\epsilon$ is well-defined, and defines on~$\gr \Gamma$ a structure of mixed Lie algebra~\cite[II.1.2.5]{Lazard}.
Note that $[\symb x, \symb y]$ is not necessarily equal to $\symb(xyx^{-1}y^{-1})$, and $\epsilon(\symb x)$ is not necessarily equal to~$\symb(x^p)$.

A $p$-filtration is called a $p$-valuation if $\omega(x) > 1/(p-1)$ for all~$x \in \Gamma$, and $\omega(x^p) = \varphi(\omega(x)) = \omega(x) + 1$ for all~$x \in \Gamma \setminus \{1\}$~\cite[III.2.1.2]{Lazard}.
If this is the case, then $\epsilon$ defines on $\gr \Gamma$ a structure of free $\bF_p[\epsilon]$-module, and the bracket is $\bF_p[\epsilon]$-bilinear.
A $p$-valuation on~$\Gamma$ is called saturated if $\Gamma$ is separated and complete in the filtration topology, 
and $\omega(x) > p/(p-1)$ implies that~$x$ is a $p$-th power in $\Gamma$~\cite[III.2.1.5, III.2.1.6]{Lazard}.

\subsection{Group cohomology.}
Let~$\Gamma$ be a profinite group, and let~$M$ be a $\bZ_p$-linearly topological, complete $\bZ_p$-module (i.e.\ $M$ is complete, and its topology has a neighborhood basis at zero given
by open $\bZ_p$-submodules) with a continuous action of~$\Gamma$.
For example, $M$ could be $k$, $\cO_K$ or~$K$ with trivial action of~$\Gamma$.
Then we write $H^i(\Gamma, M)$ for the $i$-th continuous cohomology group of~$\Gamma$ with coefficients in~$M$.
This is denoted $H^i_c(\Gamma, M)$ in~\cite{Lazard}, and is defined as the $i$-th cohomology of the complex of (inhomogeneous) cochains of~$\Gamma$ with coefficients in~$M$.
We therefore have an isomorphism 
\begin{equation}\label{universal coefficients I}
H^i(\Gamma, \cO_K) \otimes_{\cO_K} K \isom H^i(\Gamma, K)
\end{equation}
and a short exact sequence
\begin{equation}\label{universal coefficients II}
0 \to H^i(\Gamma, \cO_K) \otimes_{\cO_K} k \to H^i(\Gamma, k) \to H^{i+1}(\Gamma, \cO_K)[\varpi] \to 0.
\end{equation}

Now assume that~$\Gamma$ is a compact open subgroup of~$G(K)$, where~$G$ is a connected reductive group over~$K$.
Let~$\fg_K$ be the Lie algebra of~$G(K)$, viewed as a $\bQ_p$-Lie algebra.
Then~\cite[Th\'eor\`eme~2.4.10]{Lazard} gives an isomorphism
\[
H^*(\Gamma, K) \isom H^*(\fg_K, K)^\Gamma
\]
of graded $K$-algebras, where the $\Gamma$-action on $H^*(\fg_K, K)$ is the adjoint action.
However, as explained in~\cite[Section~3]{CasselmanWigner}, the adjoint action of $G(K)$ on $H^*(\fg_K, K)$ is trivial, because its kernel is open in the
locally profinite topology, and closed in the Zariski topology.
Hence we obtain an isomorphism
\begin{equation}\label{characteristic zero Lazard isomorphism}
H^*(\Gamma, K) \isom H^*(\fg_K, K).
\end{equation}

\subsection{Koszul complexes.}\label{Koszul complexes}
Let~$A$ be a commutative $k$-algebra.
Let~$V$ be a finite-dimensional $k$-vector space, and let $\delta: V \to A$ be a $k$-linear map.
Regard $A \otimes_k \wedge V$ as a graded $k$-algebra, with degree~$i$ component given by $A \otimes_k \wedge^i V$.
Then there is a unique $k$-linear map
\[
d: A \otimes_k \wedge V \to A \otimes_k \wedge V,
\]
of degree~$-1$, such that $d|_A = 0$ and $d|_V = \delta$, and which is a derivation in the sense that
$d(xx') = d(x)x' +(-1)^{\deg x}xd(x')$.
It is given by the formula
\begin{equation}\label{Koszul differential}
d(r \otimes (v_{1}\wedge \cdots \wedge v_{i})) = \sum_{k=1}^i(-1)^{k+1}r\delta(v_{k}) \otimes (v_{1}\wedge \cdots \wedge \widehat v_{k} \wedge \cdots \wedge v_{i}),
\end{equation}
and it satisfies~$d^2 = 0$.
(This is a slight generalization of~\cite[Proposition~4.3.1.2]{IllusieI}.)

For example, suppose $A = \Sym V$ and $\delta : V = \Sym^1 V \to A$ is the inclusion, and choose a basis~$\{x_0, \ldots, x_n\}$ of~$V$.
Then $\Sym V \otimes_k \wedge V$ becomes isomorphic to the Koszul complex associated to the sequence $x_0, \ldots, x_n \in k[x_0, \ldots, x_n]$
in~\cite[Tag~0621]{Stacks}.
More generally, let $I = (F_1, \ldots, F_r)$ be a homogeneous ideal in~$\Sym V$, and write~$\fm$ for the maximal homogeneous ideal.
Assume that the~$F_i$ are minimal homogeneous generators of~$I$, or equivalently that 
\[I/\fm I = k \lbar F_1 \oplus \cdots \oplus k \lbar F_r.\]
Define $\delta: I/\fm I \to \Sym V, \lbar F_i \mapsto F_i$.
Then $\Sym V \otimes_k \wedge (I/\fm I)$ is isomorphic to the Koszul complex associated to $F_1, \ldots, F_r \in \Sym V$.
Hence it follows from~\cite[Tag~062F]{Stacks} that, if~$(F_1, \ldots, F_r)$ is $\Sym V$-regular, 
then the homology of $\Sym V \otimes_k \wedge (I/\fm I)$ is isomorphic to $(\Sym(V)/I)[0]$.
More precisely, the quotient map $\Sym V \to (\Sym V)/I$ and the augmentation $\wedge (I/\fm I) \to k$ give rise to a map
\begin{equation}\label{graded augmentation}
\Sym V \otimes_k \wedge (I/\fm I) \to (\Sym(V)/I)[0]
\end{equation}
which commutes with the grading and the differential (zero on the target), and induces an isomorphism on homology.

\subsection{dg algebras.}\label{dg algebras}
In this section we review some facts about differential graded $k$-algebras.
We will use cohomological $\bZ$-gradings, and require the differentials to have degree~$+1$.
A graded-commutative 
differential graded $k$-algebra, or $k$-cdga for short, is a commutative monoid in the symmetric monoidal category $\cC$ of complexes of $k$-vector spaces.
If~$R$ is a $k$-cdga, the category $\dgMod(R)$ of left dg $R$-modules coincides with the category of left $R$-modules in~$\cC$.
Since~$\cC$ is cocomplete and closed, \cite[Lemma~2.2.2]{HSS} shows that $\dgMod(R)$ is symmetric monoidal under
the tensor product
\begin{equation}\label{tensor product of dg modules}
M \otimes_R N := \colim(M \otimes_k R \otimes_k N \rightrightarrows M \otimes_k N)
\end{equation}
where the two arrows describe the action of~$R$ on~$M$ and~$N$; the commutativity constraint in~$\cC$ is used to turn~$M$ into an $(R, R)$-bimodule.
We define an $R$-cdga to be a commutative monoid in $\dgMod(R)$.
By general principles, if~$M$ and~$N$ are $R$-cdga, then $M \otimes_R N$ has a natural structure of $R$-cdga.
As usual, we say that a morphism $f: M \to N$ in $\dgMod(R)$ is a quasi-isomorphism 
if it induces an isomorphism on cohomology.

We will be mostly interested in the case $R = \Sym (\ft^\vee[-2])$, which is a polynomial ring generated in degree~$2$, with zero differential. 
It is a $k$-cdga.
The objects of $\dgMod(R)$ are graded $R$-modules together with an $R$-linear differential of degree~$+1$.
The $R$-module underlying $M \otimes_R N$ is the usual tensor product, the grading has the property that $\deg(m \otimes n) = \deg(m)+\deg(n)$ if~$m, n$ are homogeneous, and
the differential is defined by $d(m \otimes n) = d_M(m) \otimes n +(-1)^{\deg(m)}m \otimes d_N(n)$ on homogeneous elements.
An $R$-cdga is a dg $R$-module~$M$ together with a map $\mu: M \otimes_R M \to M$ in $\dgMod(R)$ satisfying associativity, unit and commutativity constraints.  

Examples of $R$-cdga are obtained by changing the grading on the Koszul complexes
of Section~\ref{Koszul complexes}, in the following way.
If~$I$ is a homogeneous ideal in $R$, then~$R/I$ is an $R$-cdga with zero differential and the quotient grading.
The differential~\eqref{Koszul differential} defines a structure of $R$-cdga on
\begin{equation}\label{graded Koszul}
R \otimes_k \wedge (I/\fm I[1]).
\end{equation}
More precisely, $I$ inherits a grading as a submodule of~$R$, and we take the tensor product grading on~\eqref{graded Koszul}.
The differential~\eqref{Koszul differential} has then total degree~$+1$, as required.
(In contrast with Section~\ref{Koszul complexes}, note that~\eqref{graded Koszul} is not a complex
of modules over~$\Sym \ft^\vee$: its graded components are finite-dimensional $k$-vector spaces.)
If~$I$ is generated by a regular sequence, the augmentation map~\eqref{graded augmentation} becomes a quasi-isomorphism of $R$-cdga
\begin{equation}\label{dg augmentation}
    R \otimes_k \wedge (I/\fm I[1]) \to R/I.
\end{equation}
The following variant on~\eqref{graded Koszul} will also be used.
If~$J \subset R$ is a homogeneous ideal, and $I = \fm$ is the maximal ideal, then $I/\fm I[1] = \ft^\vee[-1]$, and we can take the base change of~\eqref{graded Koszul} to~$R/J$:
\begin{equation}\label{Koszul base change}
(R/J) \otimes_R (R \otimes_k \wedge \ft^\vee[-1]).
\end{equation}
Note that the underlying graded $R$-algebra of~\eqref{Koszul base change} 
is $(R/J) \otimes_k \wedge \ft^\vee[-1]$, and that the differential is uniquely determined by $d|_{R/J}$ (which is zero) and
$d|_{\ft^\vee[-1]}$ (which is a $k$-linear surjection of degree~$1$ from~$\ft^\vee[-1]$ to the degree~$2$ part of~$R/J$).

If~$P$ is an $R$-cdga of the form~\eqref{graded Koszul}, then the dg $R$-module underlying~$P$ is \emph{semi-free} in the sense of~\cite[Definition~11.4.3]{Yekutieli}.
To see this, since the differential on $R$ is zero, we need to construct an exhaustive increasing filtration~$\Fil_i(P)$ of~$P$ by dg-submodules, such that
$\gr_i(P)$ has zero differential, and is free as a graded $R$-module, for all~$i$.
For this, it suffices to define
\[
\Fil_i(P) := R \otimes_k \Fil_i \left (\wedge (I/\fm I)[1] \right ),
\]
where~$\wedge (I/\fm I)[1]$ is filtered by total degree.

Since~$P$ is semi-free, it follows from~\cite[Theorem~11.4.14, Proposition~10.3.4]{Yekutieli} that~$P$ is a $K$-flat dg $R$-module, and it follows from~\cite[Lemma~12.3.2]{Yekutieli}
that the tensor product functor $- \otimes_{R} P$ preserves quasi-isomorphisms in $\dgMod(R)$. 
Hence it also preserves quasi-isomorphisms of $R$-cdga.
In the proof of Proposition~\ref{upper bound} we will apply this functor to a quasi-isomorphism of the form~\eqref{dg augmentation}.

\subsection{Cohomology of the flag variety.}
To lighten notation, in this section we write~$G$ and~$B$ instead of~$G_k$ and~$B_k$ (i.e.\ $G_k$ is the $k$-group with the same based root datum as~$G$).
Let~$X := G/B$ and write $\cO := \cO_X$ for the structure sheaf, and~$\Omega^i := \wedge^i \Omega^1_{X/k}$.
Recall from e.g.\ \cite[Proposition~II.6.18]{Jantzenbook} that $H^i(X, \Omega^j) = 0$ if~$i \ne j$, and
$H^i(X, \Omega^i)$ is a trivial $G$-module of dimension
\begin{equation}\label{dimension of Hodge cohomology}
\dim_k H^i(X, \Omega^i) = |\{w \in W: \ell(w) = i\}|.
\end{equation}
Furthermore, we have the following connections between $H^i(X, \Omega^i)$, the cohomology of~$B$, and the Chow ring of~$X$.
\begin{lemma}\label{reduction to Hodge cohomology}
There is an isomorphism of graded $k$-algebras
\[
\bigoplus\nolimits_{i, j}H^i(B, \wedge^j (\fg/ \fb)^\vee) \cong \bigoplus\nolimits_{i, j} H^i(G/B, \Omega^j).
\]
\end{lemma}
\begin{proof}
    Write~$\mL(-)$ for the functor~\cite[Section~I.5.8]{Jantzenbook} from rational $B$-modules to $G$-equivariant coherent sheaves on $G/B$. 
    Then $\mL(\fg/\fb)$ is the tangent bundle of~$G/B$, and so $\mL(\wedge^j (\fg/\fb)^\vee) = \Omega^j$.
    We thus obtain from~\cite[Proposition~I.5.12]{Jantzenbook} that
    \[
    R^i \ind_B^G \wedge^j(\fg/\fb)^\vee = H^i(G/B, \Omega^j),
    \]
    and so it has trivial $G$-action.
    By~\cite[Proposition~I.4.5(b)]{Jantzenbook}, there is a spectral sequence
    \[
    E_2^{r, s} := H^r(G, R^s \ind_B^G \wedge^j(\fg/\fb)^\vee) \Rightarrow H^{r+s}(B, \wedge^j(\fg/\fb)^\vee).
    \]
    Since $H^r(G, k) = 0$ for~$r > 0$, by \cite[Corollary~II.4.11]{Jantzenbook}, we have $E_2^{r, s} = 0$ for~$r \ne 0$.
    Hence the edge map $H^s(B, \wedge^j (\fg/\fb)^\vee) \to H^s(G/B, \Omega^j)$ is an isomorphism for all~$s, j$.
    Since the edge maps are compatible with cup products,
    this concludes the proof.
\end{proof}

\begin{lemma}\label{cycle class isomorphism}
    Let~$A(G/B)$ be the Chow ring of $G/B$, graded by~$\bZ_{\geq 0}$.
    There is a graded $k$-algebra isomorphism (the cycle class map)
    \[
    \lbar A(G/B) := A(G/B) \otimes_{\bZ} k \to \bigoplus\nolimits_i H^i(G/B, \Omega^i).
    \]    
\end{lemma}
\begin{proof}
See~\cite[Section~3]{Srinivas}.
\end{proof}

\begin{lemma}\label{Borel presentation}
    Let $r := \dim_k \ft$, and assume~$p \geq h$.
    Then there is a graded $k$-algebra surjection 
    \begin{equation}\label{eqn:Borel presentation}
    \Sym \ft^\vee[-1] \cong k[X_1, \ldots, X_r] \to \lbar A(G/B)
    \end{equation}
    whose kernel is generated by homogeneous elements $F_1, \ldots F_r$ such that $\deg(F_i) = m_i+1$.
\end{lemma}
\begin{proof}
Let~$S := \Sym_\bZ X^*(T)$.
As explained in~\cite[Section~4.6]{DemazureSchubert}, there is a ring homomorphism
\[
c: S \to A(G/B)
\]
whose kernel is generated by the homogeneous invariants of~$W$ of positive degree, and whose cokernel is annihilated by the torsion index of~$G$.
If $S^W \to \bZ$ is the augmentation map, it follows that~$c$ factors through an injection $S \otimes_{S^W} \bZ \to A(G/B)$, and so $c \otimes_{\bZ} k$
can be written as a composition
\[
S \otimes_\bZ k\to (S \otimes_\bZ k) \otimes_{S^W \otimes_{\bZ} k} k \xrightarrow{\alpha} \lbar A(G/B).
\]
Since the root system of~$G$ is simply connected and irreducible, 
by~\cite[Proposition~8]{Demazureinvariants} all prime divisors of the torsion index are contained in $\{2, 3, 5\}$, and so $p$ is coprime to the torsion index, since~$p > h$.
Hence $c \otimes_\bZ k$ is surjective, and so~$\alpha$ is surjective.
By~\cite[Th\'eor\`eme~3]{Demazureinvariants}, $S^W \otimes_{\bZ} k$ is a graded polynomial ring, 
generated by homogeneous elements~$F_1, \ldots, F_r$ such that $\deg(F_i) = m_i+1$.
(In fact, \emph{loc.\ cit.} shows that the degrees of the homogeneous generators are the same as in characteristic zero, provided that~$p$
is coprime to the torsion index of~$G$.)
Furthermore, the total $k$-dimension of $(S \otimes_{\bZ} k) \otimes_{S^W \otimes_{\bZ} k} k$ is~$|W|$ (by e.g.~\cite[Th\'eor\`eme~2(c)]{Demazureinvariants}), 
which is the same as the total $k$-dimension of $\lbar A(G/B)$, by Lemma~\ref{cycle class isomorphism}
and~\eqref{dimension of Hodge cohomology}.
This implies that~$\alpha$ is an isomorphism and concludes the proof of the lemma.
\end{proof}

\section{Graded Lie algebras.}\label{graded Lie algebras}
In this section we recall the Moy--Prasad filtration on the pro-$p$ Iwahori subgroups of~$G(K)$,
and we compute its associated graded Lie algebra.
We adopt without further comment the notation of Section~\ref{background}, so that~$G$ is an almost-simple, simply connected group over~$K$ with a fixed Chevalley system, $h$ is the Coxeter number
of the root system of~$G$, and $p > he+1$.

\subsection{Moy--Prasad filtrations.}
Let~$x \in C \subset V$ be the barycenter of the dominant base alcove, so that
\[
\alpha(x) = 1/h \text{ for all } \alpha \in \Delta.
\]
Following~\cite[Lemma~6.1.6]{KPbook}, we associate to~$x$ the valuation of root datum defined by
\[
v_{\alpha, x}: U_\alpha(K) \to \frac 1 h \bZ_{\geq 0} \cup \{+\infty\}, \;\;\; v_{\alpha, x}(u) := v_\alpha(u)+\alpha(x) = v_\alpha(u) + \height(\alpha)/h
\]
for all~$\alpha \in \Phi$.
For all~$i \in \bZ_{\geq 0}$ we write
\[
U_{\alpha, x, i/h} := v_{\alpha, x}^{-1}[i/h, +\infty].
\]
Similarly, as in \cite[(13.1.1)]{KPbook}, we introduce a filtration of~$T(K)$ by letting~$T(K)_0$ be the unique maximal compact subgroup of~$T(K)$, and
\[
T(K)_{i/h} := \{t \in T(K)_0: v(\chi(t)-1) \geq i/h \text{ for all } \chi \in X^*(T)\} \text{ if } i > 0.
\]
Note that by definition
\[
T(K)_{i/h} = T(K)_{\ceiling*{i/h}}.
\]
If~$n \in \bZ_{\geq 0}$, the group $T(K)_n$ admits the following alternative description:
the evaluation map yields an isomorphism
\[
T(K) \cong X_*(T) \otimes_{\bZ} K^\times,
\]
and then $T(K)_n = X_*(T) \otimes_{\bZ} (1+\varpi^n \cO_K)$.
Finally, recall from~\cite[Definition~13.2.1]{KPbook} the groups
\[
G_{x, i/h} := \langle T(K)_{i/h}, U_{\alpha, x, i/h}: \alpha \in \Phi \rangle
\]
for~$i \in \bZ_{\geq 0}$.
These groups admit the following Iwahori-type decomposition.

\begin{lemma}\label{MP decomposition}
Let~$i \in \bZ_{> 0}$.
Then the product map (under any ordering of the factors)
\[
\prod_{\alpha \in \Phi^-} U_\alpha\left (\varpi^{\ceiling*{\frac{i-\height(\alpha)}{h} }}\cO_K \right ) \times T(K)_{\ceiling*{\frac i h}} \times \prod_{\alpha \in \Phi^+} U_\alpha\left (\varpi^{\ceiling*{\frac{i-\height(\alpha)}{h} }}\cO_K \right ) \to G_{x, i/h}
\]
is a bijection.
Furthermore, for all~$\alpha \in \Phi$ and all $z \in \cO_K, w \in 1 + \varpi \cO_K$ we have
\begin{gather*}
u_\alpha(z) \in G_{x, i/h} \text{ if and only if } v(z)+\frac{\height (\alpha)}{h} \geq \frac i h\\
\alpha^\vee(w) \in G_{x, i/h} \text{ if and only if } v(w-1) \geq \frac i h.
\end{gather*}
\end{lemma}
\begin{proof}
    The second claim is a direct consequence of the first, which we now prove.
    By~\cite[Proposition~13.2.5]{KPbook}, the product map is a bijection
    \[
      \prod_{\alpha \in \Phi^-} U_{\alpha, x, i/h} \times T(K)_{\ceiling*{\frac i h}} \times \prod_{\alpha \in \Phi^+} U_{\alpha, x, i/h} \to G_{x, i/h},
    \]
    so it suffices to prove that 
    \[U_{\alpha, x, i/h} = U_\alpha\left (\varpi^{\ceiling*{\frac{i-\height(\alpha)}{h} }}\cO_K \right ),\]
    where for any~$j$ we write $U_\alpha(\varpi^j\cO_K):= v_\alpha^{-1}[j, +\infty]$.
    This is a consequence of the fact that, if $u \in U_\alpha(K)$, then $v_{\alpha, x}(u) \geq i/h$ if and only if 
    \[
    v_\alpha(u) \geq \ceiling*{\frac{i-\height(\alpha)}{h}}.\qedhere
    \]
\end{proof}
Lemma~\ref{MP decomposition} implies that $I := G_{x, 0}$ is an Iwahori subgroup of~$G$, $I_1:= G_{x, 1/h}$ is the pro-$p$ Sylow subgroup of~$I$, and
\[
r/h \mapsto G_{x, r/h}
\]
is the Moy--Prasad filtration of~$I$.
We rescale this filtration as follows; we need the rescaling to have $\omega(x^p) = \omega(x)+1$.

\begin{defn}\label{MP filtration}
If~$z \in I_1 \setminus \{1\}$, let $\omega(z) \in (he)^{-1}\bZ$ be the unique number such that $z \in G_{x, e\omega(z)}\setminus G_{x, e\omega(z)+1/h}$.
Equivalently, $\omega(z) := \max\{t \in (he)^{-1}\bZ: z \in G_{x, et}\}$.
\end{defn}

It follows from Lemma~\ref{MP decomposition} that
\begin{equation}\label{values of omega}
\omega(u_\alpha(\varpi^i)) = \height(\alpha)/he+i/e, \;\;\; \omega(\alpha^\vee(1+\varpi^j)) = i/e, 
\end{equation}
for all~$\alpha \in \Phi, i \in \bZ_{\geq 0}, j \in \bZ_{>0}$.
Furthermore, given $x_\alpha \in U_\alpha(K) \cap I_1$ and $t \in T(K) \cap I_1$, we have
\[
\omega \left (\prod_{\alpha \in \Phi^-}x_\alpha \cdot t \cdot \prod_{\alpha \in \Phi^+} x_\alpha \right ) = \min\{\omega(t), \omega(x_\alpha) : \alpha \in \Phi\},
\]
under any ordering in the product, by the bijectivity part of Lemma~\ref{MP decomposition}.
Hence~\cite[Proposition~3.4]{LahiriSorensen} applies and shows that~$\omega$ is a saturated 
$p$-valuation on~$I_1$, with values in $(he)^{-1}\bZ_{>0}$.
(Note that the valuation on~$K$ in~\cite{LahiriSorensen} is normalized to give $v(p) = 1$, whereas for us~$v(p) = e$.)

We will write~$\Fil^\bullet(I_1)$ for the resulting filtration of~$I_1$, so that
$\Fil^{i/he}(I_1) = G_{x, i/h}$. 
The associated graded $\gr(I_1)$ is therefore an $(he)^{-1}\bZ_{>0}$-graded $\bF_p$-Lie algebra, 
with an operator $\epsilon$ defining a structure of $\bF_p[\epsilon]$-Lie algebra on~$\gr(I_1)$.
Note also that
\begin{equation*}\label{finite torus}
I \cong T(k) \ltimes I_1
\end{equation*}
where~$T(k) = X_*(T) \otimes_\bZ k$ and the map $T(k) \to I$ is induced by the Teichm\"uller lift.
The following is the main result of this section.

\begin{thm}\label{graded MP}
    Assume that $he < p-1$.
    Then~$\omega$ is a saturated, $I$-invariant $p$-valuation on~$I_1$, $\gr(I_1)$ is naturally a $k[\epsilon]$-Lie algebra, and 
    there is an isomorphism $\gr(I_1) \cong \tld \fg_e$ of graded $k[\epsilon]$-Lie algebras, which 
    intertwines the action of~$T(k) \subset I$ on~$\gr(I_1)$ with the restriction to $T(k)$ of the $T$-action on $\tld \fg_e$.
\end{thm}

\begin{rk}\label{remark on p-filtrations}
    As previously recalled, this $p$-valuation on~$I_1$ has also been studied in~\cite[Proposition~3.4]{LahiriSorensen}, 
    which works with split connected reductive $K$-groups, but does not compute
    the associated graded Lie algebra.
    Since we need this piece of information in Section~\ref{cohomology}, we have provided the details.
\end{rk}

The following corollary to Theorem~\ref{graded MP} will be the main tool in Section~\ref{cohomology}.
To state it, recall that whenever~$\fh$ is a Lie algebra over~$k$, and $V$ is a $k$-representation of~$\fh$, we can construct
a semidirect product Lie algebra $\fh \ltimes V$ by regarding~$V$ as an abelian Lie algebra, and the representation $\fh \to \End_k(V)$ as a Lie
algebra homomorphism $\fh \to \operatorname{Der}_k(V)$.
Concretely, 
\[
\text{$[H, v] := H.v$ for all~$H \in \fh, v \in V$.}
\]
Since~$\fn$ acts on~$\fg$ via the adjoint representation, we can therefore form the semidirect product $\fn \ltimes (\fg/\fn)$.
Another way of describing this product is to introduce, for all~$X, Y \in \fg$, the projections~$[X, Y]_{\fb^-}$ and~$[X, Y]_\fn$ of~$[X, Y]$ 
under the triangular decomposition $\fg = \fb^- \oplus \fn$.
Then $\fg/\fn = \fb^-$, with the $\fn$-action given by $X.B = [X, B]_{\fb^-}$ for all~$X \in \fn, B \in \fb^-$.

\begin{corollary}\label{presentation}
Assume that~$e = 1$ and~$h<p-1$.
Then there is an isomorphism of $\bF_p$-Lie algebras
\[
\gr(I_1) \otimes_{\bF_p[\epsilon]} \bF_p \cong \fn \ltimes (\fg/\fn).
\]
\end{corollary}
\begin{proof}
By Theorem~\ref{graded MP} we have
\[
\gr(I_1) \otimes_{\bF_p[\epsilon]} \bF_p \cong \tld \fg_1/v\tld \fg_1. 
\]
If~$X \in \fg[v]$ we write~$\lbar X$ for its image under mod~$v$ reduction $\fg[v] \to \fg$.
Similarly, if $Y = vB \in v \fb^-[v]$ then we write~$Y^* := \lbar B$.
Then there is an $\bF_p$-linear homomorphism 
\[
i: \tld \fg_1 = v \fb^-[v] \oplus \fn[v] \to \fb^- \rtimes \fn, (Y, X) \mapsto (Y^*, \lbar X)
\]
whose kernel is precisely $v \tld \fg_1$.
So it suffices to check that~$i$ is a Lie algebra homomorphism.

Since $[v B_1, v B_2] =v^2[B_1, B_2]$ for all~$B_1, B_2 \in \fb^-[v]$, 
we have $[vB_1, vB_2]^* = 0$, hence 
$i$ restrict to a Lie algebra homomorphism on $v \fb^-[v]$.
By definition, it also restricts to a homomorphism on $\fn[v]$.
Now, if $X \in \fn[v]$ and~$Y =vB \in v\fb^-[v]$ then 
\[
[Y, X] = v[B, X]_{\fb^-[v]} + v[B, X]_{\fn[v]},
\]
and so $i[Y, X] = \lbar{[B, X]}_{\fb^-[v]} = [\lbar B, \lbar X]_{\fb^-} = [i(Y), i(X)]_{\fb^-}$, as desired.
\end{proof}

\subsection{Proof of Theorem~\ref{graded MP}.}
As already observed, the fact that~$\omega$ is a saturated $p$-valuation is~\cite[Proposition~3.4]{LahiriSorensen}.
We begin our analysis of~$\gr(I_1)$ by computing an $\bF_p$-basis.
Bearing in mind~\eqref{values of omega}, we make the following definitions:
\begin{gather*}
\text{ if } \alpha \in \Phi^+, i \in \bZ_{\geq 0}, \text{ then } v^iX_\alpha := \symb(u_\alpha(\varpi^i)) \in \gr(I_1)_{\height(\alpha)/he + i/e}\\
\text{ if } \alpha \in \Phi^-, i \in \bZ_{\geq 1}, \text{ then } v^iX_\alpha := \symb(u_\alpha(\varpi^{i})) \in \gr(I_1)_{\height(\alpha)/he+i/e}\\
\text{ if } \alpha \in \Phi, i \in \bZ_{\geq 1}, \text{ then } v^iH_\alpha := \symb(\alpha^\vee(1+\varpi^{i})) \in \gr(I_1)_{i/e}.
\end{gather*}
Recall from Lemma~\ref{alternative presentation of tld g_e} that for all~$i \geq 1, \alpha \in \Phi$ 
such that $\height(\alpha) \equiv i \text{ mod } h$,
we write 
\[
i_\alpha := (i-\height(\alpha))/h \in \bZ_{\geq 0}.
\]
\begin{lemma}\label{basis of associated graded}
    Let~$i \in \bZ_{\geq 1}$. 
    The $\bF_p$-vector space $\gr(I_1)_{i/he}$ is naturally a $k$-vector space, with basis
    \begin{gather*}
    \{v^{i_\alpha}X_\alpha: \height(\alpha) \equiv i \text{ mod } h\} \text{ if } h \nmid i\\
    \{v^{i/h}H_\alpha: \alpha \in \Delta\} \text{ if } h \mid i.
    \end{gather*}
\end{lemma}
\begin{proof}
    By definition, $\gr(I_1)_{i/he} = G_{x, i/h}/G_{x, (i+1)/h}$.
    Since $(i+1)/h \leq 2i/h$, Lemma~\ref{MP decomposition} implies that the product map
    \begin{equation}\label{MP isomorphism}
    \prod_{\alpha \in \Phi^-} U_{x, \alpha, i/h}/U_{x, \alpha, (i+1)/h} \times T(K)_{i/h}/T(K)_{(i+1)/h} \times \prod_{\alpha \in \Phi^+} U_{x, \alpha, i/h}/U_{x, \alpha, (i+1)/h} \to G_{x, i/h}/G_{x, (i+1)/h}
    \end{equation}
    is an isomorphism of abelian groups 
    (compare the proof of~\cite[Theorem~13.5.1]{KPbook}, and note that all elements of~$\Phi$ are non-divisible).

    Assume first that $h \mid i$.
    Then $U_{x, \alpha, i/h} = U_{x, \alpha, (i+1)/h}$ for all~$\alpha \in \Phi$, and so the left-hand side of~\eqref{MP isomorphism} coincides with 
    \[
    T_{i/h}/T_{i/h+1} \cong X_*(T) \otimes_{\bZ} (1+\varpi^{i/h})/(1+\varpi^{i/h+1})
    \]
    and therefore is a $k$-vector space with basis given by the principal symbols of $\alpha^\vee(1+\varpi^{i/h})$ for~$\alpha \in \Delta$.
    
    Assume now that $h \nmid i$.
    Then Lemma~\ref{MP decomposition} shows that the left-hand side of~\eqref{MP isomorphism} coincides with
    \[
    \prod_{{\height(\alpha) \equiv i \text{ mod } h}}U_\alpha(\varpi^{\frac{i-\height(\alpha)}{h}})/U_\alpha(\varpi^{\frac{i-\height(\alpha)}{h}+1}) = 
    \prod_{{\height(\alpha) \equiv i \text{ mod } h}}U_\alpha(\varpi^{i_\alpha})/U_\alpha(\varpi^{i_\alpha+1}).
    \]
    This concludes the proof because $U_\alpha(\varpi^{i_\alpha})/U_\alpha(\varpi^{i_\alpha+1})$ 
    is a one-dimensional $k$-vector space spanned by~$\symb(u_\alpha(\varpi^{i_\alpha}))$.\qedhere
\end{proof}

Lemma~\ref{basis of associated graded} implies that $\gr(I_1)$ can be given the structure of a free $k[v]$-module with basis
\begin{equation}\label{eqn:basis of associated graded}
\{X_\alpha: \alpha \in \Phi^+\} \cup \{vX_{-\alpha}: \alpha \in \Phi^+\} \cup\{vH_\alpha: \alpha \in \Delta\},
\end{equation}
and then Lemma~\ref{alternative presentation of tld g_e} shows that the map
\begin{gather*}
f: \gr(I_1) \to \tld \fg_e\\
X_\alpha \mapsto X_\alpha \otimes 1, \;\;\; vX_{-\alpha} \mapsto X_{-\alpha} \otimes v, \;\;\; vH_\alpha \mapsto H_\alpha \otimes v
\end{gather*}
is an isomorphism of $(he)^{-1}\bZ$-graded $k[v]$-modules.
It has the additional property that it intertwines the $T(k)$-action on~$\gr(I_1)$ with the restriction to~$T(k)$ of the $T$-action on $\tld \fg_e$.
However, we do not yet know that $\gr(I_1)$ is a $k[v]$-Lie algebra.
This is implied by the next proposition.

\begin{pp}\label{computation of associated graded}
The map~$f : \gr(I_1) \to \tld \fg_e$ is a $k$-Lie algebra isomorphism.
\end{pp}
\begin{proof}
This is a direct computation by cases.
To begin with, if~$i, j \in \bZ_{\geq 1}$ and~$\alpha, \beta \in \Delta$ then
\[
0 = [v^iH_\alpha, v^j H_\beta] = [f(v^iH_\alpha), f(v^j H_\beta)],
\]
as desired.

Now let~$\alpha \in \Delta, \beta \in \Phi$, choose $i \in \bZ_{\geq 1}, j \in \bZ_{\geq \delta_{\beta<0}}$, and consider $[v^iH_\alpha, v^j X_\beta]$.
We need to prove that it coincides with $\langle \alpha^\vee, \beta \rangle v^{i+j}X_{\beta}$.
By definition, $[v^iH_\alpha, v^j X_\beta]$ is the image in $\gr(I_1)_{\height(\beta)/he + (i+j)/e}$ of the commutator $(\alpha^\vee(1+\varpi^i), u_\beta(\varpi^j))$.
Since~$u_\beta$ is part of a Chevalley system, we find that
\[
    (\alpha^\vee(1+\varpi^i), u_\beta(\varpi^j)) = u_\beta((1+\varpi^i)^{\langle \alpha^\vee, \beta \rangle}\varpi^j-\varpi^j) \equiv 
    u_\beta(\langle \alpha^\vee, \beta \rangle \varpi^{i+j}) \text{ mod } U_\beta(\varpi^{i+j+1}\cO_K).
\]
Since $U_\beta(\varpi^{i+j+1}\cO_K) \subset \Fil^{\height(\beta)/he + (i+j+1)/e}(I_1)$, we conclude that
\[
    [v^iH_\alpha, v^j X_\beta] = \symb u_\beta(\langle \alpha^\vee, \beta \rangle \varpi^{i+j}) = \langle\alpha^\vee, \beta\rangle v^{i+j}X_\beta,
\]
as desired.

Now let $\alpha, \beta \in \Phi$ be non-proportional roots, and choose $i \in \bZ_{\geq \delta_{\alpha<0}}, j \in \bZ_{\geq \delta_{\beta<0}}$.
We need to prove that $[v^i X_\alpha, v^j X_\beta] = N_{\alpha, \beta}v^{i+j}X_{\alpha+\beta}$.
The Chevalley commutator relations say that
\[
(u_\alpha(\varpi^i), u_\beta(\varpi^j)) = \prod_{\substack{p, q \in \bZ_{>0}\\p\alpha+q\beta \in \Phi}}u_{p\alpha+q\beta}(c_{p, q; \alpha, \beta}\varpi^{pi+qj})
\]
for some $c_{p, q; \alpha, \beta} \in \bZ$, such that~$c_{1, 1; \alpha, \beta} = N_{\alpha, \beta}$.
(See~\cite[Formula~(2.9.4)]{KPbook} as well as~\cite[Sections~3.2.3, 3.2.4]{BTII} for these facts.)
Since
\[
\omega \left (u_{p\alpha+q\beta}(c_{p, q; \alpha, \beta}\varpi^{pi+qj}) \right ) \geq \frac{p\height \alpha + q \height \beta}{he} + \frac{pi+qj}{e},
\]
and the right-hand side is minimized at~$(p, q) = (1, 1)$, we see that
\[
(u_\alpha(\varpi^i), u_\beta(\varpi^j)) \equiv u_{\alpha+\beta}(c_{1, 1; \alpha, \beta}\varpi^{i+j}) \text{ mod } \Fil^{\frac{\height \alpha + \height \beta}{he} + \frac{i+j}{e} + \frac{1}{he}}(I_1),
\]
which concludes the proof since $\symb u_{\alpha+\beta}(c_{1, 1; \alpha, \beta}\varpi^{i+j}) = N_{\alpha, \beta}v^{i+j}X_{\alpha+\beta}$.

Finally, choose $\alpha \in \Phi^+, i \in \bZ_{>0}, j \in \bZ_{\geq 1}$.
Since $\height(\alpha)+\height(-\alpha) = 0$ and~$\omega$ is a $p$-valuation, we know in advance that
\[
    (u_\alpha(\varpi^i), u_{-\alpha}(\varpi^j)) \in \Fil^{\frac{i+j}{e}}(I_1) = G_{x, i+j},
\]
and we need to prove that
\[
(u_\alpha(\varpi^i), u_{-\alpha}(\varpi^j)) \equiv \alpha^\vee(1+\varpi^{i+j}) \text{ mod } \Fil^{\frac{i+j}{e}+\frac{1}{he}}(I_1) = G_{x, i+j+1/h}.
\]
Since
\[
T_{i+j}/T_{i+j+1/h} \isom G_{x, i+j}/G_{i+j+1/h},
\]
it therefore suffices to prove that the torus part in the Iwahori decomposition of $(u_\alpha(\varpi^i), u_{-\alpha}(\varpi^j))$, which is an element of $T_{i+j}$,
is congruent to $\alpha^\vee(1+\varpi^{i+j})$ modulo~$T_{i+j+1}$.
Using the homomorphism $\eta_\alpha: \SL_2 \to G$ arising from our fixed Chevalley system, whose image is $\langle U_\alpha(K), U_{-\alpha}(K) \rangle$, 
we can assume without loss of generality that $G = \SL_2$.
Here we see that
\[
(u_\alpha(\varpi^i), u_{-\alpha}(\varpi^j)) = \begin{pmatrix}
    1+\varpi^{i+j}+\varpi^{2(i+j)} & - \varpi^{2i+j}\\
    \varpi^{i+2j} & 1-\varpi^{i+j}
\end{pmatrix}.
\]
The torus part of this matrix is congruent to $\diag(1+\varpi^{i+j}, 1 -\varpi^{i+j})$ modulo~$\varpi^{2(i+j)}$.
Since~$i+j\geq 1$, this concludes the proof.
\end{proof}

\begin{lemma}\label{computation of epsilon}
    Let~$\lambda \in k^\times$ be the image of $p/\varpi^e$.
    Then the operator~$\epsilon : \gr(I_1) \to \gr(I_1)$ is multiplication by~$\lambda v^e$.
\end{lemma}
\begin{proof}
Since~$\epsilon$ is $\bF_p$-linear, it suffices to prove that it coincides with~$\lambda v^e$ on the basis~\eqref{eqn:basis of associated graded}.
Since $\omega$ is a saturated $p$-valuation, we have
\[
\omega(\alpha^\vee(1+\varpi^i)^p) = \varphi(\omega(\alpha^\vee(1+\varpi^i)))
\]
and
\begin{gather*}
\omega(u_\alpha(\varpi^i)^p) = \varphi(\omega (u_\alpha(\varpi^i))).
\end{gather*}
(This can also be checked from~\eqref{values of omega}.)
Hence
\begin{gather*}
\epsilon(v^i H_\alpha) = \symb (\alpha^\vee(1+\varpi^{i})^p) = \symb(\alpha^\vee(1+p\varpi^i)) = \lambda v^{i+e}H_\alpha\\
\epsilon(v^i X_\alpha) = \symb(u_\alpha(\varpi^i)^p) = \lambda \symb(u_\alpha(\varpi^{i+e})) = \lambda v^{i+e}X_\alpha,
\end{gather*}
as desired.
(In the first equality, we have used the fact that $v(p\varpi^i) < v(\varpi^{ip})$, which follows from our assumption that $he<p-1$.)
\end{proof}

Proposition~\ref{computation of associated graded} and Lemma~\ref{computation of epsilon} imply that the map 
$f: \gr(I_1) \to \tld \fg_e$ is a $T(k)$-equivariant isomorphism 
of $(he)^{-1}\bZ_{>0}$-graded
$\bF_p[\epsilon]$-Lie algebras (in fact, it is even $k[v]$-linear).
This concludes the proof of Theorem~\ref{graded MP}.

\section{Iwahori cohomology.}\label{cohomology}
Throughout this section we assume that~$e = 1$ and~$p>h+1$, and use notation from Section~\ref{background} and Section~\ref{graded Lie algebras}.
The group~$G$ is still assumed to be almost-simple and simply connected.
We write
\begin{equation}\label{definition of gbar}
\lbar \fg_k := \gr(I_1) \otimes_{\bF_p[\epsilon]} \bF_p \cong \fn_k \ltimes (\fg_k/\fn_k),
\end{equation}
where the isomorphism holds by Corollary~\ref{presentation}.
This is a $k$-Lie algebra, with an action of~$T$ (i.e.\ a grading by weights). 
The isomorphism~\eqref{definition of gbar} intertwines the action of~$T(k)$ on $\gr(I_1)$ with the restriction of the $T$-action on~$\lbar \fg_k$ to~$T(k)$.
By~\cite[Theorem~1.1]{SorensenHochschild} there is a spectral sequence
\begin{equation}\label{Lazard spectral sequence}
E_1^{s, t}:= H^{s, t}(\lbar \fg_k, k) \Rightarrow H^{s+t}(I_1, k).
\end{equation}
The bigrading on the $E_1$-page is induced by the $h^{-1}\bZ_{>0}$-grading on~$\lbar \fg_k$, see~\cite[Definition~4.5]{SorensenHochschild}, and it has the property that
\[
\bigoplus_{s+t = n}H^{s, t}(\lbar \fg_k, k) = H^{n}(\lbar \fg_k, k),
\]
where~$H^n(\lbar \fg_k, k)$ denotes the $\bF_p$-Lie algebra cohomology of~$\lbar \fg_k$.
It is therefore isomorphic to $H^*(\lbar \fg_k \otimes_{\bF_p} k, k)$, viewed now as $k$-Lie algebra cohomology.

The spectral sequence~\eqref{Lazard spectral sequence} is multiplicative,
and functorial for automorphisms of~$(I_1, \omega)$.
It is therefore equivariant for the action of the finite torus $T(k)$ on its terms.
Since the group $T(k)$ is $p$-coprime, we conclude that there is a multiplicative spectral sequence
\begin{equation}\label{Lazard spectral sequence II}
E_1^{s, t} = H^{s, t}(\lbar \fg_k, k)^{T(k)} \Rightarrow H^{s+t}(I, k).
\end{equation}

\begin{rk}\label{multiplicative convergence}
Let~$H^*$ be a graded-commutative, finite-dimensional graded $k$-algebra.
By definition, $E_r^{ij} \Rightarrow H^*$ is a multiplicative spectral sequence if all pages $(E_r, d_r)$ are $k$-cdga's, and the isomorphism $E_{r+1} \cong H(E_r, d_r)$
preserves the multiplication.
Convergence of the spectral sequence to~$H^*$ means that the induced filtration on $H^*$ is multiplicative (i.e.\ $\Fil^i. \Fil^j \subset \Fil^{i+j}$)
and furthermore $\gr H^* \isom E_\infty$ as graded $k$-algebras.

In the special case in which $E_\infty \cong \Lambda_k(\Phi)^{\otimes f}$ (where we give~$E_\infty$ the total grading), this implies that $H^* \cong \Lambda_k(\Phi)^{\otimes f}$.
To see this, choose lifts $[x_{2m_i+1}] \in H^{2m_i+1}$ of the generators of~$\Lambda_k(\Phi)^{\otimes f}$.
Since $H^*$ is graded-commutative, and~$\Lambda_k(\Phi)^{\otimes f}$ is free graded-commutative, 
these lifts produce a morphism $\alpha: \Lambda_k(\Phi)^{\otimes f} \to H^*$ of graded $k$-algebras.
To prove that~$\alpha$ is an isomorphism, it suffices to prove that it is surjective.
Let~$L$ be the image of~$\alpha$, and endow it with the $k$-subspace filtration.
Then $\gr L \to \gr H^*$ is surjective by construction.
Hence $L = H^*$, as desired.
\end{rk}

We are going to prove that~\eqref{Lazard spectral sequence II} degenerates at~$E_1$, and to compute the algebra structure on $H^*(I, k)$.
We begin by making a reduction to the case~$f = 1$, and so we write $\lbar \fg := \fn \ltimes (\fg/\fn)$ for the $\bF_p$-form
of~$\lbar \fg$ arising from the split $\bQ_p$-form of~$G$.
Since~$\lbar \fg_k$ is a $k$-Lie algebra, there is a decomposition
\[
\lbar \fg_k \otimes_{\bF_p} k \cong \prod_{i=0}^{f-1} \lbar \fg_k^{(i)}
\]
as product of Lie ideals, where $\lbar \fg_k^{(i)} := \lbar \fg_k \otimes_{k, \Frob_p^i} k$.
The action of~$T$ on~$\lbar \fg_k^{(i)}$ is also twisted by~$\Frob_p^i$, in the sense that all degrees on~$\lbar \fg_k$ are multiplied by~$p^i$.
The K\"unneth formula for Lie algebra cohomology then gives a $T$-equivariant isomorphism of graded $k$-algebras
\begin{equation}\label{Kunneth formula for gbar}
H^*(\lbar \fg_k, k) \cong \bigotimes_{i=0}^{f-1}H^*(\lbar \fg, \bF_p)^{(i)} \otimes_{\bF_p} k.
\end{equation}
We now determine the factors of~\eqref{Kunneth formula for gbar}.

\begin{lemma}\label{Hochschild--Serre for semidirect products}
The Hochschild--Serre spectral sequence
\[
E_2^{i, j} = H^i(\fn, H^j(\fg/\fn, \bF_p)) \cong H^i(\fn, \wedge^j (\fg/\fn)^\vee) \Rightarrow H^{i+j}(\lbar \fg, \bF_p)
\]
degenerates at~$E_2$.
\end{lemma}
\begin{proof}
    The Hochschild--Serre spectral sequence for Lie algebra cohomology with trivial coefficients always degenerates at~$E_2$ 
for semidirect products with abelian kernels.
See for instance~\cite[Theorem~1.2]{DPLie}.
\end{proof}

Together with Lemma~\ref{Hochschild--Serre for semidirect products}, our assumption that~$p > h+1$ implies that~\eqref{Lazard spectral sequence II} can be simplified as follows.

\begin{lemma}\label{weights of exterior algebra}
Let~$n>h(1+p+\cdots+p^{f-1})$ be an integer.
Choose~$\lambda_0, \ldots, \lambda_{f-1} \in X^*(T)$, and let $\lambda := \sum_{i=0}^{f-1}\lambda_ip^i$.
Assume that~$\lambda_i \in \wedge \fg$ for all~$i$, and $\lambda \in n X^*(T)$.
Then~$\lambda = 0$.
\end{lemma}
\begin{proof}
We generalize the argument in~\cite[Section~2.2(1)]{AJinduced}, which is the case of this lemma when~$f = 1$ and~$n$ is prime.
Since the weights of $\wedge \fg$ form a $W$-stable subset of $X^*(T)$, and every $W$-orbit contains a dominant weight, we can assume without loss of generality that~$\lambda$
is dominant.
Hence $\lambda = n\mu$ for some dominant~$\mu$.
Since $\lambda_i$ is a weight of $\wedge \fg$, it is a sum of pairwise distinct roots, and so $\lambda_i \leq 2 \rho$.
This implies $\lambda \leq 2(1+p+\cdots+p^{f-1})\rho$.
If~$\alpha_0^\vee$ is the highest coroot, we conclude that 
\[
0 \leq n\langle\mu, \alpha_0^\vee \rangle \leq 2(1+p+\cdots+p^{f-1}) \langle\rho, \alpha_0^\vee\rangle = 2(1+p+\cdots+p^{f-1})(h-1)<2n.
\]
Hence $\langle\mu, \alpha_0^\vee\rangle \in \{0, 1\}$.
If~$\langle\mu, \alpha_0^\vee\rangle = 0$, then~$\mu = 0$, and we conclude that~$\lambda = 0$, as desired.

There remains to prove that $\langle\mu, \alpha_0^\vee\rangle = 1$ leads to a contradiction.
As explained in the proof of~\cite[Section~2.2(1)]{AJinduced}, this equality implies that~$\mu$ is a fundamental weight not contained in the root lattice~$\langle \bZ \cdot \Phi \rangle$.
So it suffices to check that if~$\mu$ is such a weight, then 
\[
n\mu \not \leq 2(1+p+\cdots+p^{f-1})\rho.
\]
Suppose otherwise, and choose a basis of simple roots~$\{\alpha_1, \ldots, \alpha_\ell\}$ of~$X^*(T)_\bQ$.
Since $n > h(1+p+\cdots+p^{f-1})$ and $2\rho$ has positive coefficients in this basis, 
it follows that every coefficient of~$h\mu$ is strictly smaller than the corresponding coefficient of~$2 \rho$.
This can be checked to be false on a case-by-case basis, from~\cite[Planches I--IX]{BourbakiLieIV-VI}, as we now illustrate.

For type~$A_\ell$, we have~$h = \ell+1$, the coefficient of~$\alpha_i$ in the $i$-th fundamental weight~$\varpi_i$ is $i(h-i)/h$, and the coefficient of~$\alpha_i$ in~$2\rho$ is $i(h-i)$.
For type~$B_\ell$, we have~$h = 2\ell$ and the only fundamental weight not in~$\langle \bZ \cdot \Phi \rangle$ is~$\varpi_\ell$.
The coefficient of~$\alpha_\ell$ in~$\varpi_\ell$ is $\ell/2$, and its coefficient in~$2\rho$ is~$\ell^2$.
For type~$C_\ell$, we have $h = 2\ell$, and~$\alpha_1$ has coefficient~$1$ in~$\varpi_i$ and~$2\ell$ in~$2\rho$.
For type~$D_\ell$, we have $h = 2\ell-2$, and $\alpha_1$ has coefficient~$1$ in~$\varpi_i$ for~$1 \leq i \leq \ell-2$, and $2\ell-2$ in~$2\rho$.
Finally, the coefficient of $\alpha_{\ell-1}$ and~$\alpha_\ell$ in~$2\rho$ is $\ell(\ell-1)/2$, which coincides with the coefficient of $\alpha_{\ell-1}$ in $h\varpi_{\ell-1}$, resp.
the coefficient of $\alpha_\ell$ in~$h\varpi_\ell$.
For type~$E_6$, $h = 12$, and a consideration of the coefficient of~$\alpha_1$ suffices for~$\varpi_1, \varpi_3, \varpi_5$, whereas~$\alpha_6$ suffices for~$\varpi_6$.
For type~$E_7$, $h = 18$, and a consideration of the coefficient of~$\alpha_1$ suffices for~$\varpi_2, \varpi_5$, whereas~$\alpha_7$ suffices for~$\varpi_7$.
Finally, types~$E_8, F_4$ and~$G_2$ have connectivity index~$1$, and so there is nothing to prove for these types.
This concludes the proof of the lemma.
\end{proof}

\begin{lemma}\label{weights of gbar-cohomology}
Assume that $p > h+1$.
Then 
\[
H^*(\lbar \fg_k, k)^{T(k)} = H^*(\lbar \fg_k, k)^{T} = \bigotimes_{i=0}^{f-1}H^*(\lbar \fg, \bF_p)^T \otimes_{\bF_p} k.
\]
\end{lemma}
\begin{proof}
By~\eqref{Kunneth formula for gbar}, the first equality is equivalent to
\begin{equation}\label{T(k) and T}
\left ( \bigotimes\nolimits_{i=0}^{f-1}H^*(\lbar \fg, \bF_p)^{(i)} \otimes_{\bF_p} k \right )^{T(k)} = \left ( \bigotimes\nolimits_{i=0}^{f-1}H^*(\lbar \fg, \bF_p)^{(i)} \otimes_{\bF_p} k \right )^{T}.
\end{equation}
For this, it suffices to prove that if~$\lambda_i$ is a weight of $H^*(\lbar \fg, \bF_p)$, and $\lambda := \sum_{i=0}^{f-1}\lambda_ip^i$ is divisible by~$p^f-1$,
then~$\lambda = 0$.
By Lemma~\ref{weights of exterior algebra}, it suffices to prove that $\lambda_i$ is a weight of~$\wedge \fg$ for all~$i$.
By Lemma~\ref{Hochschild--Serre for semidirect products}, $\lambda_i$ is a weight of $H^r(\fn, \wedge^s(\fg/\fn)^\vee)$ for some~$r, s$.
Hence there exists a weight~$\mu_i$ of $\wedge (\fg/\fn)^\vee$ such that $\lambda_i$ is a weight of $H^*(\fn, \mu_i)$.
By Kostant's theorem~\eqref{twisted Kostant's theorem}, we see that $\lambda_i = w_i \cdot 0 +\mu_i$ for some~$w_i \in W$.
Since $w_i \cdot 0$ is a sum of distinct negative roots, it is a weight of $\wedge(\fn^-)$, whereas
by construction, $\mu_i$ is a weight of $\wedge(\fg/\fn)^\vee \cong \wedge(\fb)$.
Hence $\lambda_i$ is a weight of $\wedge \fg$, as desired.
This concludes the proof of~\eqref{T(k) and T}.

For the second equality, there remains to prove that if~$\lambda = 0$, then~$\lambda_i = 0$ for all~$i$.
However, $\lambda = 0$ implies that~$\lambda_0 \in p X^*(T)$.
Since~$p>h$, Lemma~\ref{weights of exterior algebra} with~$f = 1, n = p$ implies that~$\lambda_0 = 0$.
Iterating this argument implies that~$\lambda_i = 0$ for all~$i$, as desired.
\end{proof}

By the first equality in Lemma~\ref{weights of gbar-cohomology}, the spectral sequence~\eqref{Lazard spectral sequence II} can be rewritten as
\begin{equation}\label{Lazard spectral sequence III}
E_1^{s, t} = H^{s, t}(\lbar \fg_k, k)^{T} \Rightarrow H^{s+t}(I, k).
\end{equation}
The following theorem is the main result of this section.

\begin{thm}\label{main theorem}
If~$p > h+1$
then~\eqref{Lazard spectral sequence III} degenerates at~$E_1$, and $H^*(I, k)$ is isomorphic to
$\Lambda_k(\Phi)^{\otimes f}$ as graded $k$-algebras.
\end{thm}
\begin{proof}
By Lemma~\ref{lower bound}~(3), $\dim_k H^*(I, k)$ is bounded below by the dimension of $\Lambda_k(\Phi)^{\otimes f}$.
By Proposition~\ref{upper bound}, if~$p > h$ then $H^*(\lbar \fg, \bF_p)^T \cong \Lambda_{\bF_p}(\Phi)$.
By the second equality in Lemma~\ref{weights of gbar-cohomology}, we deduce that $E_1 \cong \Lambda_k(\Phi)^{\otimes f}$. 
For dimension reasons, this implies that~\eqref{Lazard spectral sequence III} degenerates at~$E_1$. 
Since $H^*(I, k)$ is a graded-commutative algebra, we conclude by Remark~\ref{multiplicative convergence} that $H^*(I, k) \cong \Lambda_k(\Phi)^{\otimes f}$.
\end{proof}

The rest of this section aims to construct the isomorphism $H^*(\lbar \fg, \bF_p)^T \cong \Lambda_{\bF_p}(\Phi)$ needed in the proof of Theorem~\ref{main theorem}.
We already know from Lemma~\ref{Hochschild--Serre for semidirect products} that $H^*(\lbar \fg, \bF_p)^T \cong H^*(\fn, \wedge^*(\fg/\fn)^\vee)^T$ as graded $\bF_p$-algebras,
where the right-hand side has the total grading.
We begin by proving that
$H^*(\fn, \wedge^*(\fg/\fb)^\vee)^T$
is isomorphic as bigraded $\bF_p$-algebras 
to the Hodge cohomology of $G/B$, and so to $\lbar A(G/B)$, i.e.\ the Chow ring of~$G/B$, tensored with~$\bF_p$.

\begin{lemma}\label{correct dimension Hodge cohomology}
We have 
\[
\dim_{\bF_p} H^i(\fn, \wedge^j (\fg/\fb)^\vee)^T = \dim_{\bF_p} H^i(G/B, \Omega^j).
\]
\end{lemma}
\begin{proof}
We need to prove that
\begin{align*}
\dim_{\bF_p} H^i(\fn, \wedge^j (\fg/\fb)^\vee)^T =\;\; &|\{w \in W: \ell(w) = j\}| \text{ if~$i = j$}\\
& 0 \text{ otherwise}.
\end{align*}
By Kostant's theorem~\eqref{twisted Kostant's theorem}, for all~$\lambda \in X^*(T)$ we have $H^*(\fn, \lambda)^T = 0$ unless $\lambda = -w\cdot 0$ for some~$w \in W$, and
furthermore $H^*(\fn, -w\cdot 0)^T \cong k[-\ell(w)]$ is one-dimensional and concentrated in degree~$\ell(w)$.

The finite-dimensional $B$-module~$\wedge^j (\fg/\fb)^\vee$ admits a filtration with one-dimensional graded pieces, which are in bijection with the weights of $\wedge^j (\fg/\fb)^\vee$.
Furthermore, by~\cite[Proposition~2.2]{FPdiscrete}, we have
\begin{equation}\label{classification of shifted weights}
\dim_k \Hom_T(-w \cdot 0, \wedge^j(\fg/\fb)^\vee) = \delta_{\ell(w), j}.
\end{equation}
Hence the only weights of $\wedge^j(\fg/\fb)^\vee$ of the form $-w \cdot 0$ have~$\ell(w) = j$.
Hence, if~$i \ne j$ and $\lambda$ is a weight of~$\wedge^j(\fg/\fb)^\vee$, then $H^i(\fn, \lambda)^T = 0$.
By d\'evissage this immediately implies $H^i(\fn, \wedge^j(\fg/\fb)^\vee)^T = 0$ if~$i \ne j$, and that $\dim_k H^j(\fn, \wedge^j(\fg/\fb)^\vee)^T$ is the
number of weights of $\wedge^j(\fg/\fb)^\vee$ of the form~$-w \cdot 0$.
By~\eqref{classification of shifted weights}, this is $|\{w \in W: \ell(w) = j\}|$, as desired.
\end{proof}

\begin{lemma}\label{torus invariants}
For all~$i, j \geq 0$,
the edge map of~\eqref{ordinary to restricted} induces an isomorphism
\[
H^i(U_1, \wedge^j(\fg/\fb)^\vee)^T \to H^i(\fn, \wedge^j(\fg/\fb)^\vee)^T.
\]
\end{lemma}
\begin{proof}
The spectral sequence~\eqref{ordinary to restricted} for~$U$ acting on~$\wedge^j(\fg/\fb)^\vee$ takes the form
\[
E_1^{r, s} := H^{s-r}(\fn, \wedge^j(\fg/\fb)^\vee) \otimes_k \Sym^r(\fn^\vee)^{(1)} \Rightarrow H^{r+s}(U_1, \wedge^j(\fg/\fb)^\vee).
\] 
Since~$\wedge^j(\fg/\fb)^\vee$ is a $B$-module, this is a spectral sequence of $T$-modules.
To prove the lemma, it suffices to prove that~$(E_1^{r, s})^T = 0$ if~$r \ne 0$.
By Kostant's theorem~\eqref{twisted Kostant's theorem}, every weight of $H^{s-r}(\fn, \wedge^j(\fg/\fb)^\vee)$ has the form $w\cdot 0 + \mu$ for some weight~$\mu$ of $\wedge^j(\fg/\fb)^\vee$.
So it suffices to prove that if $\nu \in X^*(T)$, and $w \cdot 0 + \mu = p \nu$, then $\nu = 0$.
Since~$w \cdot 0 = w(\rho)-\rho$ is a sum of pairwise distinct negative roots, $w \cdot 0 + \mu$ is a weight of $\wedge \fg$, in fact of $\wedge \fn^- \otimes \wedge \fn$.
Since~$p > h$, if~$\nu \ne 0$ we obtain a contradiction to Lemma~\ref{weights of exterior algebra} with~$f = 1, n = p$.
\end{proof}

\begin{lemma}\label{B-cohomology}
For all~$j \geq 0$,
the edge map of~\eqref{Hochschild--Serre for Frobenius kernel} (i.e.\ restriction) induces an isomorphism
\[
H^j(B, \wedge^j(\fg/\fb)^\vee) \to H^j(U_1, \wedge^j(\fg/\fb)^\vee)^T.
\]
\end{lemma}
\begin{proof}
The spectral sequence~\eqref{Hochschild--Serre for Frobenius kernel} takes the form
\[
E^{r, s}_2 := H^r(B/B_1, H^s(B_1, \wedge^j (\fg/\fb)^\vee)) \Rightarrow H^{r+s}(B, \wedge^j (\fg/\fb)^\vee).
\]
We claim that $E_2^{r, s} = 0$ if~$s < j$.
This implies that the edge map
\begin{equation}\label{inclusion to prove isomorphism}
H^j(B, \wedge^j(\fg/\fb)^\vee) \to H^0(B/B_1, H^j(B_1, \wedge^j (\fg/\fb)^\vee)) \subset H^j(B_1, \wedge^j (\fg/\fb)^\vee)^T
\end{equation}
is an isomorphism.
Since $H^j(B_1, \wedge^j (\fg/\fb)^\vee)^T = H^j(U_1, \wedge^j (\fg/\fb)^\vee)^T$ via restriction (as is already true for $T_1$-invariants),
to conclude the proof of the lemma, it suffices thus to prove that the displayed inclusion in~\eqref{inclusion to prove isomorphism} is an equality,
which we do by comparing dimensions.
We begin by observing that
\[
    H^j(B_1, \wedge^j (\fg/\fb)^\vee)^T = H^j(U_1, \wedge^j (\fg/\fb)^\vee)^T = H^j(\fn, \wedge^j(\fg/\fb)^\vee)^T,
\]
where the second equality is Lemma~\ref{torus invariants}. 
This implies
\[
\dim_{\bF_p} H^j(B_1, \wedge^j (\fg/\fb)^\vee)^T = \dim_{\bF_p} H^j(G/B, \Omega^j) = \dim_{\bF_p} H^j(B, \wedge^j(\fg/\fb)^\vee),
\]
where the first equality follows from Lemma~\ref{correct dimension Hodge cohomology}, and the second equality follows from Lemma~\ref{reduction to Hodge cohomology}.
Hence the inclusion in~\eqref{inclusion to prove isomorphism} is indeed an equality.
This concludes the proof of the lemma.

There remains to prove the claim.
It suffices to prove that
\[
s < j \implies H^s(B_1, \wedge^j(\fg/\fb)^\vee) = 0.
\]
By e.g.~\cite[Section~II.9.22(1)]{Jantzenbook} or~\cite[Section~2.9]{AJinduced} (recalling that in those references, $B$ corresponds to the negative roots), 
if $H^*(B_1, \lambda) \ne 0$ then $\lambda \in -(W \cdot 0) + p X^*(T)$.
Suppose that $\lambda := -(w \cdot 0) + p \nu$ is a weight of $\wedge^j(\fg/\fb)^\vee \cong \wedge^j \fn$, and observe that $-(w \cdot 0)$ is a sum of pairwise distinct positive roots, hence it is also a weight of
$\wedge\fn$.
This implies that~$p\nu$ is a weight of $\wedge \fn^- \otimes \wedge\fn \subset \wedge \fg$
and so Lemma~\ref{weights of exterior algebra} implies that $\nu = 0$.
Hence $\lambda = -(w \cdot 0)$;
since~$\lambda$ is a weight of $\wedge^j (\fg/\fb)^\vee$, this also implies $\ell(w) = j$, by~\eqref{classification of shifted weights}.

At this point, it suffices to prove that if $\ell(w) = j$ and~$s < j$ then $H^s(B_1, -w \cdot 0) = 0$.
This is a consequence of~\cite[Section~2.9(1)]{AJinduced}, asserting that
\[
    H^s(B_1, -w \cdot 0) \otimes_{\bF_p} \lbar \bF_p \cong H^{s-\ell(w)}(B_1, \lbar \bF_p).\qedhere
\]
\end{proof}

Combining Lemmas~\ref{torus invariants} and~\ref{B-cohomology} we obtain an isomorphism of bigraded $\bF_p$-algebras
\begin{equation}\label{n to B}
\bigoplus\nolimits_{i, j} H^i(\fn, \wedge^j (\fg/\fb)^\vee)^T \cong \bigoplus\nolimits_{i, j} H^i(B, \wedge^j (\fg/ \fb)^\vee),
\end{equation}
since the edge maps are compatible with cup products.
Observe that 
the terms in~\eqref{n to B} with~$i \ne j$ vanish, by Lemma~\ref{correct dimension Hodge cohomology} and Lemma~\ref{reduction to Hodge cohomology}.
We conclude from Lemmas~\ref{reduction to Hodge cohomology} and~\ref{cycle class isomorphism} that there is an isomorphism of graded $\bF_p$-algebras
\begin{equation}\label{n to Chow}
\bigoplus\nolimits_{i} H^i(\fn, \wedge^i (\fg/\fb)^\vee)^T \cong \lbar A(G/B).
\end{equation}
We now proceed to the computation of 
\[
H^*(\fn, \wedge (\fg/\fn)^\vee)^T.
\]
We have an exact sequence of $B$-modules
\begin{equation}\label{universal extension}
0 \to (\fg/\fb)^\vee \to (\fg/\fn)^\vee \to \ft^\vee \to 0
\end{equation}
for the trivial action of $B$ on~$\ft^\vee$.
It induces a filtration on $\wedge^i (\fg/\fn)^\vee$ for all~$i$, and so a multiplicative filtration on $\wedge (\fg/\fn)^\vee$.
We have
\[
\gr^p \wedge^j(\fg/\fn)^\vee \cong \wedge^p(\fg/\fb)^\vee \otimes_{\bF_p} \wedge^{j-p} \ft^\vee,
\]
hence there is a trigraded, multiplicative spectral sequence
\begin{equation}\label{Koszul spectral sequence}
E_1^{p, q, j} := H^{p+q}(\fn, \wedge^p(\fg/\fb)^\vee)^T \otimes_{\bF_p} \wedge^{j-p} \ft^\vee \Rightarrow H^{p+q}(\fn, \wedge^j(\fg/\fn)^\vee)^T
\end{equation}
with $\deg d_1 = (1, 0, 0)$.
Lemma~\ref{correct dimension Hodge cohomology} implies that $E^{p, q, j}_1 = 0$ if~$q \ne 0$, and so~\eqref{Koszul spectral sequence} degenerates at~$E_2$.

The total grading on~$E_1^{p, 0, j}$, together with the differential~$d_1$, defines an $\bF_p$-cdga that we will denote~$(E_1, d_1)$.
We now compute a presentation of~$(E_1, d_1)$.
We will follow the conventions in Section~\ref{dg algebras}, and it will be useful to introduce some more notation.
Let $\lbar A^{(2)}(G/B)$ denote~$\lbar A(G/B)$ with all the degrees doubled.
Let $J$ be the kernel of the surjection $\Sym \ft^\vee[-2] \to \lbar A^{(2)}(G/B)$ in~\eqref{eqn:Borel presentation}, so that $\lbar A^{(2)}(G/B) = \Sym \ft^\vee[-2]/J$.
By Lemma~\ref{Borel presentation}, $J$ is generated by $r = \dim \ft^\vee$ homogeneous elements~$F_1, \ldots, F_r$ of degree $\deg(F_i) = 2m_i+2$.

\begin{lemma}\label{recognizing E_1 II}
The $\bF_p$-cdga~$(E_1, d_1)$ is isomorphic to 
\begin{equation}\label{recognized complex}
(\Sym \ft^\vee[-2]/J) \otimes_{\Sym \ft^\vee[-2]} ( \Sym \ft^\vee[-2] \otimes_{\bF_p} \wedge \ft^\vee[-1])
\end{equation}
where the second factor is~\eqref{Koszul base change}.
(In more detail, the differential on the second factor is induced by the degree-one inclusion $\ft^\vee[-1] \to \Sym \ft^\vee[-2]$.)
\end{lemma}
\begin{proof}
The underlying graded $\bF_p$-algebra of~\eqref{recognized complex} is 
\[
\Sym \ft^\vee[-2]/J  \otimes_{\bF_p} \wedge \ft^\vee[-1]
\]
with the total grading associated to the tensor product bigrading.
We begin by proving that $E_1$ is isomorphic to $\operatorname{Tot}(\Sym \ft^\vee[-2]/J  \otimes_{\bF_p} \wedge \ft^\vee[-1])$ as
graded $\bF_p$-algebras.
By~\eqref{Koszul spectral sequence} we have
\[
\bigoplus_{p+j = n}E^{p, 0, j}_1 = \bigoplus_{p+j = n} H^p(\fn, \wedge^p(\fg/\fb)^\vee)^T \otimes_{\bF_p} \wedge^{j-p} \ft^\vee = \bigoplus_{2p+k = n} H^p(\fn, \wedge^p(\fg/\fb)^\vee)^T \otimes_{\bF_p} \wedge^{k} \ft^\vee.
\]
By~\eqref{n to Chow} we obtain an isomorphism
\[
\lambda : E_1 \isom \operatorname{Tot}(\lbar A^{(2)}(G/B) \otimes_{\bF_p} \wedge \ft^\vee[-1]) = \operatorname{Tot}(\Sym \ft^\vee[-2]/J  \otimes_{\bF_p} \wedge \ft^\vee[-1])
\]
of graded $\bF_p$-algebras, as desired.

There remains to prove that $\lambda$ preserves differentials.
Note that $\lambda(E_1^{0,0,1}) = \ft^\vee[-1]$ and that $\lambda(E^{1, 0, 1}) = H^1(\fn, \wedge^1 (\fg/\fb)^\vee)^T = \ft^\vee[-2]$.
Since $\Sym \ft^\vee[-2]/J$ is generated by~$\ft^\vee[-2]$, and~$\wedge \ft^\vee[-1]$ is generated by~$\ft^\vee[-1]$, it suffices to prove that $d_1 |_{E^{1, 0, 1}} = 0$ 
and $d_1 |_{E^{0, 0, 1}}$ is
an isomorphism onto $E^{1, 0, 1}$.
Since~$\deg(d_1) = (1, 0, 0)$ and $E^{2, 0, 1} = 0$, we immediately see that $d_1 |_{E^{1, 0, 1}} = 0$.
On the other hand,
\[
E_1^{0, 0, 1} \xlongrightarrow{d_1} E_2^{1, 0, 1}
\]
is isomorphic to the connecting homomorphism~$\del$ of~\eqref{universal extension}.
So we need to prove that~$\del$ is an isomorphism, or equivalently that it is injective.
This holds because $H^0(\fn, (\fg/\fn)^\vee)^T = 0$.
In fact, if $\varphi: \fg \to \bF_p$ is a weight-zero functional,
then it vanishes on~$\fn$ and~$\fn^-$; and if $\varphi$ is $\fn$-invariant, then it vanishes on~$\ft$, since $0 = (X_\alpha \varphi)(X_{-\alpha}) = \varphi[X_\alpha, X_{-\alpha}]$.
\end{proof}

\begin{pp}\label{upper bound}
    Under our standing assumption that~$p > h+1$, there is an isomorphism
    \[
    H^*(\lbar \fg, \bF_p)^T \cong \Lambda_{\bF_p}(\Phi)
    \]
    of graded $\bF_p$-algebras.
\end{pp}
\begin{proof}
We are going to prove that $H(E_1, d_1) \cong \Lambda_{\bF_p}(\Phi)$ in the spectral sequence~\eqref{Koszul spectral sequence}.
Since~\eqref{Koszul spectral sequence} degenerates at~$E_2$ and is concentrated in~$q = 0$, this will imply that the cohomology algebra 
$H^*(\fn, \wedge^*(\fg/\fn)^\vee)^T$ is isomorphic to $\Lambda_{\bF_p}(\Phi)$, 
under the total grading.
Since $H^*(\lbar \fg, \bF_p)^T$ is isomorphic to $H^*(\fn, \wedge^*(\fg/\fn)^\vee)^T$ under the total grading, by Lemma~\ref{Hochschild--Serre for semidirect products}, 
this will imply the proposition (using Remark~\ref{multiplicative convergence} to get the algebra structure on $H^*(\lbar \fg, \bF_p)$).

By Lemma~\ref{recognizing E_1 II}
we have
\[
(E_1, d_1) \cong (\Sym \ft^\vee[-2]/J) \otimes_{\Sym \ft^\vee[-2]} (\Sym \ft^\vee[-2] \otimes_k \wedge \ft^\vee[-1]).
\]
Here we are using notation from Section~\ref{dg algebras}, and $\otimes_{\Sym \ft^\vee[-2]}$ coincides with the tensor product of $\Sym \ft^\vee[-2]$-cdga.
Note that the second factor is an instance of~\eqref{graded Koszul}.

By Lemma~\ref{Borel presentation}, the ideal~$J$ 
is generated by $\dim \ft^\vee$ homogenenous elements~$F_1, \ldots, F_r$ of degree~$2m_i+2$.
Hence $\wedge (J/\fm J[1]) = \Lambda_{\bF_p}(\Phi)$.
On the other hand, since $\Sym \ft^\vee[-2]$ is Cohen--Macaulay of Krull dimension~$r$, and $\lbar A(G/B)$ has Krull dimension zero, it follows that 
$(F_1, \ldots, F_r)$ is a regular sequence in~$\Sym \ft^\vee[-2]$.
Hence~\eqref{dg augmentation} constructs a quasi-isomorphism of $\Sym \ft^\vee[-2]$-cdga
\[
\Sym \ft^\vee[-2] \otimes_{\bF_p} \wedge(J/\fm J[1]) \to \Sym \ft^\vee[-2]/J.
\]
Similarly, there is a quasi-isomorphism of $\Sym \ft^\vee[-2]$-cdga
\[
\Sym \ft^\vee[-2] \otimes_{\bF_p} \wedge \ft^\vee[-1] \to \bF_p[0].
\]
As explained in Section~\ref{dg algebras}, if~$P$ is a $\Sym \ft^\vee[-2]$-cdga 
of the form~\eqref{graded Koszul}, then the functor $- \otimes_{\Sym \ft^\vee[-2]} P$ preserves quasi-isomorphisms
of $\Sym \ft^\vee[-2]$-cdga.
We thus have a diagram
\[ 
\begin{tikzcd}
(E_1, d_1)\\
(\Sym \ft^\vee[-2] \otimes_{\bF_p} \wedge(J/\fm J)[1]) \otimes_{\Sym \ft^\vee[-2]} (\Sym \ft^\vee[-2] \otimes_{\bF_p}\wedge\ft^\vee[-1]) \arrow[u] \arrow[d]\\
(\Sym \ft^\vee[-2] \otimes_{\bF_p} \wedge(J/\fm J)[1]) \otimes_{\Sym \ft^\vee[-2]} \bF_p[0]
\end{tikzcd}
\]
where both arrows are quasi-isomorphisms of $\Sym \ft^\vee[-2]$-cdga.
Since the differential on $\Sym \ft^\vee[-2] \otimes_{\bF_p} \wedge(J/\fm J)[1]$ become zero after $- \otimes_{\Sym \ft^\vee[-2]} \bF_p[0]$, we conclude that
\[
H(E_1, d_1) \cong \wedge (J/\fm J)[1] = \Lambda_{\bF_p}(\Phi),
\]
as desired.
\end{proof}

\section{Reductive groups.}

\subsection{Split reductive groups.}\label{reductive groups}
In this section we conclude the proof of Theorem~\ref{main theorem intro}.
We continue to assume that $e = 1$.
Let~$G$ be a split reductive group over~$K$ with a fixed Borel pair~$(B, T)$ and a fixed Chevalley system. 
Write~$\Phi := \Phi(G, T)$, and we recall from Section~\ref{Lie algebra cohomology} the definition of the graded-commutative algebra $\Lambda_k(\Phi)$.
Write~$\fg_K$ for the $K$-Lie algebra of~$\fg$, so that $\fg_K = \fz_K \oplus \fg_K^{\der}$ for a central subalgebra~$\fz_K$ of $K$-dimension equal to the $K$-rank of~$Z$.

The Chevalley system induces a model of~$G$ over~$\cO_K$, a compact open subgroup~$G(\cO_K)$, and an Iwahori subgroup~$I \subset G(\cO_K)$ 
corresponding to the $B$-positive roots $\Phi(B, T)$.
More explicitly, we have
\begin{gather*}
    G(\cO_K) = \langle U_\alpha(\cO_K), T(\cO_K): \alpha \in \Phi(G, T) \rangle\\
    I = \langle U_{\alpha}(\cO_K), T(\cO_K), U_{-\alpha}(\varpi\cO_K) : \alpha \in \Phi(B, T) \rangle.
\end{gather*}
Let~$Z$ be the connected centre, and~$G^\der$ the derived subgroup, of~$G$.
Let~$h$ be the maximum of the Coxeter numbers of irreducible components of the root system of~$G$, and assume throughout this section that~$p > h+1$.

Let~$T^{\der} := T \cap G^{\der}$, which is a maximal torus in~$G^{\der}$: this is because it contains and centralizes a maximal torus of $G^{\der}$,
which is self-centralizing in $G^{\der}$.
It follows from the Iwahori decomposition that the intersection $I^{\der} := I \cap G^{\der}(K)$ is an Iwahori subgroup of $G^{\der}(K)$ 
with pro-$p$ Sylow subgroup $I_1^{\der} := I_1 \cap G^{\der}(K)$.
It follows similarly that $Z_1 := Z(K) \cap I_1$ is the pro-$p$ Sylow subgroup of~$Z(\cO_K)$.
The finite torus~$T(k)$ acts by conjugation on~$I_1$, and the restriction map
\begin{equation}
    H^*(I, k) \isom H^*(I_1, k)^{T(k)}
\end{equation}
is an isomorphism.

\begin{lemma}\label{reduction to semisimple groups}
With notation as in the previous paragraph,
the multiplication map $Z_1 \times I_1^{\der} \to I_1$ is a group isomorphism.
\end{lemma}
\begin{proof}
By the Iwahori decomposition, it suffices to prove that the multiplication $Z_1 \times \Syl_p(T^\der(\cO_K)) \to \Syl_p(T(\cO_K))$ is an isomorphism.
Let~$r$ be the rank of~$T$.
The pullback $X^*(T) \to X^*(T^{\der}) \times X^*(Z)$ is injective, with cokernel
isomorphic to $X^*(T^{\der})/\langle \bZ \cdot \Phi \rangle$,
which is a finite abelian $p$-coprime group, by our assumption that~$p > h+1$.
The pullback is therefore isomorphic to a map $\bZ^r \to \bZ^r$ induced by a diagonal matrix with diagonal entries~$n_1, \ldots, n_r$ coprime to~$p$. 
Hence the
multiplication $Z(\cO_K) \times T^{\der}(\cO_K) \to T(\cO_K)$ is isomorphic to a map $(\cO_K^\times)^r \to (\cO_K^\times)^{r}$ which is the direct product of $x \mapsto x^{n_i}$.
Hence its kernel and cokernel are finite abelian $p$-coprime groups,
and so they become zero after localizing at~$p$, as desired. 
\end{proof}

We now compute the rational cohomology of~$I$, and deduce a lower bound on its $k$-cohomology.

\begin{lemma}\label{lower bound}
    Let $G$ be a split reductive group over~$K$ with connected centre~$Z$ and root system~$\Phi$. 
    Let~$I \subset G(K)$ be an Iwahori subgroup, and let~$Z_1 = Z(K) \cap I_1$.
    Then the following are true:
    \begin{enumerate}
    \item $H^*(I, K) \cong H^*(\fg_K, K)$.
    \item $\dim_K H^*(\fg_K, K) = \dim_k (H^*(Z_1, k) \otimes_k \Lambda_k(\Phi)^{\otimes f})$.
    \item $\dim_k H^*(I, k) \geq \dim_k (H^*(Z_1, k) \otimes_k \Lambda_k(\Phi)^{\otimes f})$.
    \end{enumerate}
    \end{lemma}
    \begin{proof}
    By~\eqref{universal coefficients I} and~\eqref{universal coefficients II} we see that $\dim_k H^*(I, k) \geq \dim_K H^*(I, K)$.
    Hence part~(3) follows from parts~(1) and~(2).
      Part~(1) is a special case of the Lazard isomorphism in characteristic zero~\eqref{characteristic zero Lazard isomorphism}; we emphasize that $H^*(\fg_K, K)$ denotes here the
      Lie algebra cohomology with $K$-coefficients of the $\bQ_p$-Lie algebra~$\fg_K$.

    We now prove part~(2).
    Writing~$\fg_{\bQ_p}$ for the split $\bQ_p$-form of~$\fg_K$, and applying~\eqref{Q_p-Lie to K-Lie}, we obtain
    \[
    H^*(\fg_K, K) \cong \bigotimes\nolimits_{i=0}^{f-1}H^*(\fg_{\bQ_p}, \bQ_p) \otimes_{\bQ_p} K
    \]
    as graded $K$-algebras.
    Hence the left-hand side of~(2) is equal to $(\dim_{\bQ_p} H^*(\fg_{\bQ_p}, \bQ_p))^f$, 
    and the K\"unneth formula applied to the decomposition $\fg_{\bQ_p} = \fz_{\bQ_p} \oplus \fg_{\bQ_p}^{\der}$ implies that
    \[
    \dim_{\bQ_p} H^*(\fg_{\bQ_p}, \bQ_p) = \dim_{\bQ_p}H^*(\fz_{\bQ_p}, \bQ_p)\cdot \dim_{\bQ_p}H^*(\fg_{\bQ_p}^{\der}, \bQ_p).
    \]
    On the other hand, since~$K$ is unramified and~$p > h+1 \geq 2$, 
    we have 
    \[
    Z_1 \cong (1+p \cO_K)^{\rank Z} \cong \cO_K^{\rank Z}.
    \]
    Hence $\dim_k H^*(Z_1, k) = 2^{f \rank Z}$, which also equals
    $\dim_{\bQ_p} H^*(\fz_{\bQ_p}, \bQ_p)^f$, since $H^*(\fz_{\bQ_p}, \bQ_p)$ is isomorphic to~$\wedge \fz_{\bQ_p}^\vee$.
    Since \eqref{cohomology of semisimple Lie algebra} shows that $H^*(\fg_{\bQ_p}^{\der}, \bQ_p) \cong \Lambda_{\bQ_p}(\Phi)$,
    this concludes the proof of part~(2).
    \end{proof}

The next lemma shows that the cohomology of Iwahori subgroups with coefficients in~$k$ does not change under central isogenies, provided that~$p > h+1$, as usual.

\begin{lemma}\label{reduction to simply connected groups}
Let~$G^{\der}$ be a split semisimple group over~$K$ with a fixed Chevalley system.
Let~$\phi: G^{\sccover} \to G^{\der}$ be a simply connected cover.
Let $T^{\sccover}$ be the preimage of~$T^{\der}$,
and let~$I^{\sccover}$ be the Iwahori subgroup of~$G^{\sccover}(K)$ arising from the Chevalley system.
Then $H^*(I^{\der}, k) \isom H^*(I^{\sccover}, k)$ via restriction.
\end{lemma}
\begin{proof}
    Arguing in the same way as in the proof of Lemma~\ref{reduction to semisimple groups}, we see that
    \[
    \phi: \Syl_p T^{\sccover}(\cO_K) \to \Syl_p T^{\der}(\cO_K)
    \]    
    is an isomorphism.
    Since $\phi : U_\alpha(K) \to \phi(U_\alpha(K))$ is an isomorphism for all~$\alpha \in \Phi$, we deduce that $\phi$ induces an isomorphism
    \[
    \phi: \Syl_p(I^{\sccover}) \to \Syl_p(I^{\der}).
    \]
    Hence $\phi$ has normal image. 
    If we write $\mu_1 := \ker(\phi: I^{\sccover} \to I^{\der})$, and $\mu_2 = \coker(\phi: I^{\sccover} \to I^{\der})$, then $\mu_1$ and~$\mu_2$ are finite $p$-coprime groups,
    and there is an exact sequence
    \[
    0 \to \mu_1 \to I^{\sccover} \xrightarrow{\phi} I^{\der} \to \mu_2 \to 0.
    \]
    A computation with the Hochschild--Serre spectral sequence shows that the inflation map $H^*(\phi(I^{\sccover}), k) \to H^*(I^{\sccover}, k)$ is an isomorphism, 
    and the restriction map
    \[
    H^*(I^{\der}, k) \to H^*(\phi(I^{\sccover}), k)^{\mu_2}
    \]
    is an isomorphism.
    To conclude, it suffices therefore to prove that 
    \[
    \dim_k H^*(I^{\der}, k) \geq \dim_k H^*(I^{\sccover}, k).
    \]
    Theorem~\ref{main theorem} together with the K\"unneth formula, applied to the decomposition of $I^{\sccover}$ according to direct almost-simple factors of~$G^{\sc}$,
    implies that $H^*(I^{\sccover}, k) \cong \Lambda_k(\Phi)^{\otimes f}$.
    Hence it suffices to prove that 
    \[
    \dim_k H^*(I^{\der}, k) \geq \dim_k \Lambda_k(\Phi)^{\otimes f}.
    \]
    This follows from Lemma~\ref{lower bound}~(3) applied to the group~$G^{\der}$, which has trivial~$Z_1$.
\end{proof}

Finally, the next two propositions compute the cohomology of~$I$ and~$G(\cO_K)$.

\begin{pp}\label{pp: Iwahori cohomology for reductive groups}
Let $G$ be a split reductive group over~$K$ with root system~$\Phi$ and connected centre~$Z$.
Let~$I \subset G(K)$ be an Iwahori subgroup arising from a choice of Chevalley system in~$G$, and let~$Z_1 = Z(K) \cap I_1$.
Assume that~$e = 1$ and~$p > h+1$.
Then
\[
H^*(I, k) \cong H^*(Z_1, k) \otimes_k \Lambda_k(\Phi)^{\otimes f}.
\]
\end{pp}
\begin{proof}
    Lemma~\ref{reduction to semisimple groups} implies that
    \begin{equation}\label{factorization of Iwahori cohomology}
    H^*(I, k) = (H^*(Z_1, k) \otimes_k H^*(I^{\der}_1, k))^{T(k)} = H^*(Z_1, k) \otimes_k H^*(I^{\der}_1, k)^{T(k)}
    \end{equation}
    since the $T(k)$-action on~$Z_1$ is trivial.
    There remains to prove that 
    \[
    H^*(I_1^{\der}, k)^{T(k)} \cong \Lambda_k(\Phi)^{\otimes f}.
    \]
    Note that $H^*(I_1^{\der}, k)^{T^{\der}(k)}$ is the cohomology of an Iwahori subgroup of $G^{\der}$, which by Lemma~\ref{reduction to simply connected groups}
    and Theorem~\ref{main theorem} is isomorphic to $\Lambda_k(\Phi)^{\otimes f}$.
    So it suffices to prove
    that the inclusion 
    \[
    H^*(I_1^{\der}, k)^{T(k)} \to H^*(I_1^{\der}, k)^{T^{\der}(k)}
    \]
    is an isomorphism.
    This is a consequence of Lemma~\ref{lower bound}~(3), which together with~\eqref{factorization of Iwahori cohomology} implies that
    \[
        \dim_k H^*(I^{\der}_1, k)^{T(k)} \geq \dim_k \Lambda_k(\Phi)^{\otimes f} = \dim_k H^*(I_1^{\der}, k)^{T^{\der}(k)}.\qedhere
    \]
\end{proof}

\begin{pp}\label{pp: cohomology of the maximal compact subgroup}
Let $G$ be a split reductive group over~$K$ with a fixed Chevalley system, and let~$I \subset G(\cO_K)$ be an Iwahori subgroup.
Assume that $e = 1$ and~$p > h+1$.
Then the restriction map $H^*(G(\cO_K), k) \to H^*(I, k)$ is an isomorphism.
\end{pp}
\begin{proof}
The composition
\[
H^*(G(\cO_K), k) \xrightarrow{\res} H^*(I, k) \xrightarrow{\operatorname{cores}} H^*(G(\cO_K), k)
\]
is multiplication by the index~$(G(\cO_K):I)$, which is coprime to~$p$: in fact, $I$ is the preimage of a Borel subgroup of $G(k)$ under the mod~$p$ reduction map
$G(\cO_K) \to G(k)$.
Hence the restriction map is injective.
To see that it is an isomorphism, it suffices therefore to prove that 
\[
\dim_k H^*(G(\cO_K), k) \geq \dim_k H^*(I, k).
\]
By Proposition~\ref{pp: Iwahori cohomology for reductive groups} and Lemma~\ref{lower bound}~(2), the right-hand side is equal to $\dim_K H^*(\fg_K, K)$.
Since $\dim_k H^*(G(\cO_K), k) \geq \dim_K H^*(G(\cO_K), K)$, it therefore suffices to prove that $H^*(G(\cO_K), K) \cong H^*(\fg_K, K)$.
This is a special case of~\eqref{characteristic zero Lazard isomorphism}.
\end{proof}

\subsection{Central division algebras.}\label{Morava}
Continue to assume that~$K/\bQ_p$ is an unramified extension of degree~$f$, and choose~$n < p-1$.
Write~$q := p^f$.
Let~$D/K$ be a central division algebra of invariant~$1/n$, with maximal order~$\cO_D$, maximal ideal~$\fm_D$, and residue field~$k_D$.
Let~$U^i_D = 1+\fm_D^i$, and recall that~$U^1_D$ is the unique pro-$p$ Sylow subgroup of~$\cO_D^\times$, and that $\cO_D^\times \cong k_D^\times \ltimes U^1_D$ via the Teichm\"uller lift.

Let~$I \subset \GL_{n}(K)$ be the upper-triangular Iwahori subgroup, let $T \subset \GL_n$ be the diagonal maximal torus, and write 
$\varepsilon_i \in X^*(T) \cong \bZ^n$ for the $i$-th standard basis vector.
Following the discussion below Lemma~\ref{alternative presentation of tld g_e}, we put $\fg := kZ \oplus \mathfrak{sl}_{n, k}$, where~$Z$ is a central element.
Since $p > n+1$, this is actually isomorphic to~$\mathfrak{gl}_{n, k}$.
Let~$\fn \subset \fg$ be the upper-triangular nilpotent subalgebra, let $\fb^{-} \subset \fg$ 
be the lower-triangular subalgebra, and let
$\tld \fg_1 := v\mathfrak{b}^-[v] \oplus \fn[v]$ be as in Section~\ref{Lie algebras} (over the base~$k$).
It has an $n^{-1}\bZ$-grading, such that multiplication by~$v$ has degree~$1$, and the central element~$Z$ has degree~$1$.
We will work with the uniformizer~$p$ of~$\cO_K$, hence we put~$\epsilon := v$.

Recall that $\tld \fg_1$ and~$\fg[v]$ also have an $X^*(T)$-grading, arising from the root decomposition of~$\fg$. 
By Lemma~\ref{reduction to semisimple groups} and Theorem~\ref{graded MP}, there is a saturated, $I$-invariant $p$-valuation on~$I_1$ admitting an isomorphism
\[
\gr I_1 \isom \tld \fg_1
\]
of $n^{-1}\bZ$-graded $k[\epsilon]$-Lie algebras,
such that the $T(k)$-action on the left-hand side is induced by the $X^*(T)$-grading on the right-hand side.
More precisely, there is an isomorphism
\begin{equation}\label{characters of finite split torus}
\chi: X^*(T)/(q-1)X^*(T) \isom \Hom(T(k), k^\times)
\end{equation}
such that $\chi(\varepsilon_i) : T(k) = (k^\times)^n \to k^\times$ is projection to the $i$-th factor.
The isomorphism $\gr I_1 \isom \tld \fg_1$ is then equivariant with respect to~$\chi$, in the sense that for any weight~$\lambda$, 
the group~$T(k)$ acts on the preimage of the $\lambda$-graded part of~$\tld \fg_1$ by the character~$\chi(\lambda)$.

There is the following analogue of~\eqref{characters of finite split torus} for the group $D^\times$.
Let~$w \in S_n$ be the $n$-cycle~$(1\cdots n)$.
Then there is an isomorphism
\begin{equation}\label{characters of finite nonsplit torus}
\chi_{D}: X^*(T)/(qw-1)X^*(T) \isom \Hom(k_D^\times, k_D^\times)
\end{equation}
such that 
$\chi_D(\varepsilon_i) : k_D^\times \to k_D^\times$ is $x \mapsto x^{q^{i-1}}$.

The next lemma produces an alternative presentation of~$\tld \fg_1$, which will be useful in computing with~$\cO_D^\times$.
If $i, j$ are integers, we will write $j\oplus i$ for the unique integer in $\{1, \ldots, n\}$ congruent to $j+i$ modulo~$n$.

\begin{lemma}\label{alternative model for gbar}
Let $k^n := \prod_{i=1}^{n}k$, and let~$\Phi$ be the left shift operator on~$k^n$.
Then there are $k$-linear isomorphisms $\lambda_i: \gr^{i/n} \tld \fg_1 \to k^n$ such that
\begin{enumerate}
\item $\epsilon : \gr^{i/n} \tld \fg_1 \to \gr^{i/n+1} \tld \fg_1$ corresponds to the identity.
\item $\lambda_{i+j}([x,y]) = \lambda_i(x)\Phi^i(\lambda_j y)-\lambda_j(y)\Phi^j(\lambda_i x)$ for all $x \in \gr^{i/n}\tld \fg_1, y \in \gr^{j/n}\tld \fg_1$.
\item $T$ acts on the the preimage of the $j$-th factor of $k^n$ by $\varepsilon_j-\varepsilon_{j\oplus i}$.
\end{enumerate}
\end{lemma}
\begin{proof}
Recall that we write~$\fg[i]$ for the summand of~$\fg$ of Coxeter degree~$i$, so that
\[
\fg[i] := \bigoplus_{j=1}^n kE_{j, j\oplus i},
\]
where~$E_{a, b}$ denotes the elementary matrix with~$1$ in the $(a, b)$-entry, and zero elsewhere.
By Corollary~\ref{alternative presentation of tld g},
there are isomorphisms $\gr^{i/n}\tld \fg_1 \to \fg[i]$, 
which are equivariant for the Lie bracket and the action of~$T$, and for which $\epsilon$ corresponds to the identity.
If we compose these isomorphisms with 
\[
\fg[i] \isom k^n, E_{j, j\oplus i} \mapsto e_j,
\]
where~$e_j$ is the $j$-th standard basis vector of $k^n$,
then part~(1) is true.
Part~(2) is then a rephrasing of the identity $[E_{r,s}, E_{u,v}] = \delta_{s,u} E_{r,v}-\delta_{v,r} E_{u,s}$, and part~(3) is the statement that~$E_{j,j\oplus i}$ 
has weight $\varepsilon_j-\varepsilon_{j\oplus i}$.
\end{proof}

\begin{lemma}\label{gbar for D*}
There is a saturated, $\cO_D^\times$-invariant $p$-valuation on $U^1_D$ such that $\gr U^1_D$ is a graded $k[\epsilon]$-Lie algebra,
\[
(\gr U^1_D) \otimes_{k} k_D \cong \tld \fg_1 \otimes_{k} k_D
\] 
as graded $k_D[\epsilon]$-Lie algebras, and the $k_D^\times$-action on the left-hand side is induced by the $T$-grading on $\tld \fg_1$ via~\eqref{characters of finite nonsplit torus}.
\end{lemma}
\begin{proof}
As explained in~\cite[Section~6.3]{SorensenHochschild}, the congruence filtration $\Fil^{i/n} U^1_D = U^i_D$ is a saturated, $k_D^\times$-invariant $p$-valuation,
and there are isomorphisms
$\gr^{i/n} U^1_D \to k_D$ 
such that 
\begin{enumerate}
\item $\epsilon : \gr^{i/n} U^1_D \to \gr^{i/n+1} U^1_D$ corresponds to the identity,
\item if $x \in \gr^{i/n} U^1_D$ and~$y \in \gr^{j/n} U^1_D$, then
\[
[x, y] = xy^{q^i}- yx^{q^j},
\]
\item $k_D^\times$ acts on $\gr^{i/n} U^1_D$ via the character $k_D^\times \to k_D^\times, x \mapsto x^{1-q^i}$.
\end{enumerate}
Now
\[
(\gr^{i/n} U^1_D) \otimes_{k} k_D = (\gr^{i/n}U^1_D) \otimes_{k_D} (k_D \otimes_{k} k_D) \cong (\gr^{i/n}U^1_D) \otimes_{k_D} k_D^n,
\]
where the last isomorphism is induced by $\mu: k_D \otimes_k k_D \to k_D^n$ with $i$-th component $x \otimes y \mapsto x^{q^i}y$.
Composing with $\gr^{i/n} U^1_D \isom k_D$, we obtain isomorphisms
\[
\lambda_i: \gr^{i/n} U^1_D \otimes_k k_D \isom k_D^n
\]
which send $\lambda_{i+j}([x, y] \otimes 1)$ to $\lambda_i(x\otimes 1)\Phi^i(\lambda_j(y\otimes 1))-\lambda_j(y\otimes 1)\Phi^j(\lambda_i(x\otimes 1))$.
Furthermore, if~$z \in k_D^\times$ then $\lambda_i((z.x) \otimes 1) = \mu(z^{1-q^i} \otimes 1) \lambda_i(x \otimes 1)$.
Hence the lemma is a consequence of the presentation 
of~$\tld \fg_1$ described in Lemma~\ref{alternative model for gbar}.
\end{proof}

\begin{pp}\label{pp: cohomology of D*}
Assume that~$p > n+1$, let~$\Phi$ be the root system of type~$A_{n-1}$, and let~$Z_1$ be the pro-$p$ Sylow subgroup of~$\cO_K^\times$.
Then the following graded $k_D$-algebras are isomorphic:
\[H^*(\cO_D^\times, k_D) \cong (k_D[x_1]/x_1^2 \otimes_{k_D} \Lambda_{k_D}(\Phi))^{\otimes f} \cong H^*(Z_1, k_D) \otimes_{k_D} \Lambda_{k_D}(\Phi)^{\otimes f} \cong H^*(I, k_D).\]
\end{pp}
\begin{proof}
As in~\eqref{Lazard spectral sequence III}, there is a multiplicative spectral sequence
\begin{equation}\label{nonsplit Lazard spectral sequence}
E_1^{r, s} := H^{r, s}(\gr U^1_D \otimes_{\bF_p[\epsilon]} \bF_p, k_D)^{k_D^\times} \Rightarrow H^{r+s}(\cO_D^\times, k_D).
\end{equation}
Let $\iota: k \to k_D$ be the inclusion.
We have
\[
\gr(U^1_D) \otimes_{\bF_p} k_D \cong \prod_{i=0}^{f-1}(\gr(U^1_D) \otimes_{k, \iota \circ \Frob_p^i} k_D) = \prod_{i=0}^{f-1}(\gr(U^1_D) \otimes_{k, \iota} k_D) \otimes_{k_D, \Frob_p^i} k_D
\]
as graded $k_D[\epsilon]$-Lie algebras, and as $k_D^\times$-modules.
Regarding a $k_D[k_D^\times]$-module as a $\Hom(k_D^\times, k_D^\times)$-graded $k_D$-vector space, we conclude that
\[
\gr(U^1_D) \otimes_{\bF_p} k_D \cong \prod_{i=0}^{f-1}(\gr(U^1_D) \otimes_{k, \iota} k_D)^{(i)},
\]
where, as usual, the index~$(i)$ means that the degrees are multiplied by~$p^i$ (and not~$q^i$).
In the rest of this proof, we omit~$\iota$ from the tensor product notation.
By Lemma~\ref{gbar for D*}, we find that
\[
\gr(U^1_D) \otimes_{\bF_p} k_D \cong \prod_{i=0}^{f-1}(\tld \fg_1 \otimes_k k_D)^{(i)},
\]
as graded $k_D[\epsilon]$-Lie algebras, 
and the $\Hom(k_D^\times, k_D^\times)$-grading on the left-hand side is induced by the $X^*(T)$-grading on the right-hand side via~\eqref{characters of finite nonsplit torus}.
The $E_1$-page of~\eqref{nonsplit Lazard spectral sequence} is thus isomorphic to 
\[
\left (\bigotimes_{i=0}^{f-1} H^*(\tld \fg_1/\epsilon \tld \fg_1, k)^{(i)} \otimes_{k} k_D \right )^{(qw-1)X^*(T)}. 
\]
Writing~$\fg^{\der} := \mathfrak{sl}_{n, k}$, Theorem~\ref{presentation} shows that there is a $T$-equivariant isomorphism 
\[
\tld \fg_1 / \epsilon \tld \fg_1 \cong kZ \oplus (\fn \ltimes (\fg^{\der}/\fn)) \cong \fn \ltimes (\fg/\fn),
\]
for the trivial $T$-action on the central element~$Z$.
Hence the $E_1$-page of~\eqref{nonsplit Lazard spectral sequence} is isomorphic to
\[
\left (\bigotimes_{i=0}^{f-1} H^*(\fn \ltimes (\fg/\fn), k)^{(i)} \otimes_{k} k_D \right )^{(qw-1)X^*(T)}
\]
and by Lemma~\ref{weights of exterior algebra in nonsplit case}, 
this coincides with the weight zero part, i.e.\ with
\[
\left (\bigotimes_{i=0}^{f-1} H^*(\fn \ltimes (\fg/\fn), k)^{(i)} \otimes_{k} k_D \right )^{T}.
\]
By the second inequality in Lemma~\ref{weights of gbar-cohomology}, this coincides with
\[
\bigotimes_{i=0}^{f-1} H^*(\fn \ltimes (\fg/\fn), k)^{T} \otimes_k k_D.
\]
We thus deduce from Proposition~\ref{upper bound} that
\[
E_1 \cong (k_D[x_1]/x_1^2 \otimes_{k_D} \Lambda_{k_D}(\Phi))^{\otimes f}.
\]
Note that the right-hand side is isomorphic to $H^*(Z_1, k_D) \otimes_{k_D} \Lambda_{k_D}(\Phi)^{\otimes f}$.
We can now conclude by the same argument as Theorem~\ref{main theorem}.
More precisely, by Remark~\ref{multiplicative convergence}, it suffices to prove that~\eqref{nonsplit Lazard spectral sequence} degenerates at~$E_1$.
Let~$K_D$ be the unramified extension of~$K$ with residue field~$k_D$.
Since 
\[
\Lie (\cO_D^\times) \otimes_{\bQ_p} K_D = D \otimes_{\bQ_p} K_D \cong \prod_{i=0}^{f-1}\mathfrak{gl}_{n, K_D},
\]
the isomorphism~\eqref{characteristic zero Lazard isomorphism} shows that
\[
H^*(\cO_D^\times, K_D) \cong (K_D[x_1]/x_1^2 \otimes_{K_D} \Lambda_{K_D}(\Phi))^{\otimes f}.
\] 
Hence the dimension of $H^*(\cO_D^\times, k_D)$ is at least the dimension of~$E_1$, which forces~\eqref{nonsplit Lazard spectral sequence} to degenerate at $E_1$.
\end{proof}

\begin{lemma}\label{weights of exterior algebra in nonsplit case}
Let~$\fg := \mathfrak{gl}_{n, k}$, and assume~$p > n$.
Let~$\lambda$ be a weight of  
\[
\bigotimes_{i=0}^{f-1} H^*(\fn \ltimes (\fg/\fn), k)^{(i)} \otimes_{k} k_D,
\]
and assume that $\chi_D(\lambda) = 0$,
where~$\chi_D$ is defined in~\eqref{characters of finite nonsplit torus}.
Then~$\lambda = 0$.
\end{lemma}
\begin{proof}
The assumption means that there exist weights $\{\lambda_i : 0 \leq i \leq f-1\}$ of $H^*(\fn \ltimes (\fg/\fn), k)$ such that $\lambda = \sum_{i=0}^{f-1} \lambda_ip^i$ and $\chi_D(\lambda) = 0$.
As in the proof of Lemma~\ref{weights of gbar-cohomology}, 
Kostant's theorem~\eqref{twisted Kostant's theorem} implies that for all~$i$ there exist~$w_i \in W$ and a weight~$\mu_i$ of $\wedge \fb$ such that $\lambda_i = w_i \cdot 0 + \mu_i$.
Hence~$\lambda_i$ is a weight of~$\wedge \fg$, and so
it suffices to prove that if~$\lambda_i$ is a weight of $\wedge \fg$, and $\chi_D(\sum_{i=0}^{f-1}\lambda_i p^i) = 0$, then~$\lambda = 0$.
This is an analogue of Lemma~\ref{weights of exterior algebra}, and could presumably be treated in a similar way, but we give a simplified argument for~$\mathfrak{gl}_n$.
Without loss of generality, $\lambda_i$ is a sum of pairwise distinct roots. 
The roots have the form~$\varepsilon_a - \varepsilon_b$ for~$1 \leq a \ne b \leq n$.
Hence, if we write
\[
\lambda_i = \sum_{j=1}^n\lambda_{i, j}\varepsilon_j,
\]
with~$\lambda_{i, j} \in \bZ$, then~$|\lambda_{i, j}|\leq n -1 < p-1$ for all~$i, j$.
This implies that if 
\[
\lambda^* := \sum_{\substack{0 \leq i \leq f-1\\1 \leq j \leq n}}\lambda_{i, j}p^{i+(j-1)f}
\]
is divisible by~$q^n-1$ then~$\lambda = 0$.
This concludes the proof because by definition (see~\eqref{characters of finite nonsplit torus}) we have
\[
\chi_D(\lambda): x \mapsto x^{\lambda^*}.\qedhere
\]
\end{proof}

\bibliographystyle{amsalpha}
\bibliography{mathsrefs}
\end{document}